\documentclass[10pt]{amsart}
\usepackage{bbm,amsthm,amsmath,amsfonts,amscd,amssymb}
\usepackage[all, cmtip]{xy}
\usepackage{cancel}
                           
\usepackage[pdftitle={Lie-Nijehuis bialgebroids}, pdfauthor={Thiago Drummond}]{hyperref}                               
\numberwithin{equation}{section}

\newtheorem{thm}{Theorem}[section]
\newtheorem{lem}[thm]{Lemma}
\newtheorem{prop}[thm]{Proposition}
\newtheorem{cor}[thm]{Corolary}
\newtheorem{exam}[thm]{Example}
\newtheorem{dfn}[thm]{Definition}
\newtheorem{rem}[thm]{Remark}

\newcommand{\Der}          {{\mathrm {Der}}}

\newcommand{\Lie}       {\mathcal{L}}
\newcommand{\frakx}     {\mathfrak{X}}

\newcommand{\G}         {\mathcal{G}}
\newcommand{\toto}      {\rightrightarrows}

\renewcommand{\t}      {\mathsf{t}}

\newcommand{\C}            {\mathbb{C}}
\newcommand{\R}            {\mathbb{R}}

\newcommand{\E}            {\mathcal{E}}

\newcommand{\<}   {\langle}
\renewcommand{\>}   {\rangle}

\newcommand{\ct}        {{r,T^*}}
\newcommand{\GDer}      {\mathrm{GDer}}
\newcommand{\D}     {\mathcal{D}}

\begin{document}
\title[]
{Lie-Nijenhuis bialgebroids.}
\author[]{Thiago Drummond}
\address{Departamento de Matem\'atica, Instituto de Matem\'atica,
Universidade Federal do Rio de Janeiro,
Caixa Postal 68530, Rio de Janeiro, RJ, 21941-909, Brasil.
}
\email{drummond@im.ufrj.br}

\begin{abstract}
 We introduce Lie-Nijenhuis bialgebroids as Lie bialgebroids endowed with an additional derivation-like object. They give a complete infinitesimal description of Poisson-Nijenhuis groupoids, and key examples include Poisson-Nijenhuis manifolds, holomorphic Lie bialgebroids and flat Lie bialgebra bundles. To achieve our goal we develop a theory of ``generalized derivations'' and their duality, extending the well-established theory of derivations on vector bundles. 
\end{abstract}

\maketitle

\section{Introduction}
Since the seminal work of F. Magri and M. Morosi \cite{MaMo}, Poisson-Nijenhuis (PN) geometry has become an important area of research due to its
relationship to bi-Hamiltonian systems and with various other geometric structures (particularly, Lie bialgebroids \cite{Kos} and holomorphic Poisson geometry \cite{LMX2}). Its rich connection with Lie theory was clear since the early days \cite{Kos, KM} and it has grown deeper with lots of aspects being studied, e.g. integration to symplectic Nijenhuis groupoids \cite{StX}, symplectic realization \cite{Pet}, multiplicative integrable systems \cite{Bo}. The present paper fits into this context and its main contribution is the introduction of the concept of Lie-Nijenhuis bialgebroids. They give a complete description of the infinitesimal data associated to PN groupoids. Important examples of Lie-Nijenhuis bialgebroids come from PN manifolds, holomorphic Lie bialgebroids and flat Lie bialgebra bundles such as those studied in \cite{AN}.


It is well-known that a Poisson structure on a smooth manifold $M$ endows $T^*M$ with a Lie algebroid structure in such a way that $(TM, T^*M)$ is a Lie bialgebroid \cite{Mac-Xu}. The guiding principle in the paper is to describe PN geometry as some extra geometric data on the bialgebroid $(TM, T^*M)$ in a way that can be extended to arbitrary Lie bialgebroids. We show that this additional structure is encoded in a higher degree generalization of vector bundle derivations, and we develop the theory of these ``generalized derivations'' in this paper.

While a derivation on a vector bundle $E \to M$ is a pair $(\Delta, X)$, where $X \in \frakx(M)$ is a vector field and $\Delta: \Gamma(E) \to \Gamma(E)$ is a $\R$-linear operator satisfying the Leibniz equation
$$
\Delta(fu) = f\Delta(u) + (\Lie_{X}f) u,
$$
a \textit{generalized derivation of degree $k$} consists of a triple $\D = (D, l, r)$, where $r \in \Omega^k(M, TM)$,  $l \in \Omega^{k-1}(M, \mathrm{End}(E))$ and $D: \Gamma(E) \to \Omega^k(M, E)$ satisfying the Leibniz type equation
$$
D(fu) = fD(u) + df \wedge l(u) - \<df, r\> \otimes u,
$$
for $u \in \Gamma(E),\, f \in C^{\infty}(M)$.
Generalized derivations were studied before in \cite{BD} - although they were not named therein. A key fact established in that paper is that there is a correspondence between generalized derivations of degree $k$ on $E$ and elements of $\Omega^k(E, TE)$ satisfying an extra linearity condition. Moreover, since the space of these linear vector-valued forms are closed under the Fr\"olicher-Nijenhuis bracket \cite{BD2}, the space $\GDer(E)$ of generalized derivations has a graded Lie algebra structure. In degree zero, $\GDer^0(E)$ is the Lie algebra of derivations on $E$ and the correspondence above recovers the well known relationship between derivations and linear vector fields on $E$ (see \cite[\S 3.4]{Mac-book} and references therein).

A fundamental aspect of generalized derivations introduced in this paper is that there is graded Lie algebra isomorphism $\GDer(E) \to \GDer(E^*)$ which extends the dualization of derivations and the corresponding dualization of linear vector fields (see Thereom \ref{thm:duality}). In degree 1, which is the most relevant case for the study of PN geometry, the dual of $\D=(D,l,r) \in \GDer^1(E)$ is  $\D^\top=(D^\top, l^*, r) \in \GDer^1(E^*)$, determined by 
$$
\<D^\top_X(\mu), u\> = \Lie_X\<\mu, l(u)\> - \Lie_{r(X)}\<\mu, u\> - \<\mu, D_X(u)\>, 
$$
for  $u \in \Gamma(E), \, \mu \in \Gamma(E^*)$ and  $X \in \frakx(M)$. 

When the vector bundle is a Lie algebroid $A \to M$, one can impose additional compatibility conditions between generalized derivations and the Lie algebroid structure, e.g. in degree 0, one can ask for the vector bundle derivation to be a derivation of the Lie bracket. In general, this compatibility condition is expressed by a set of equations called \textit{IM equations} in the paper \cite{BD}. 

A \textit{Lie-Nijenhuis bialgebroid} is a Lie bialgebroid $(A, A^*)$ equipped with an element $\D \in \GDer^1(A)$ such that $[\D, \D]=0$ and both $\D$ and its dual $\D^\top$ satisfy the IM equations. The relationship with PN geometry is via a characterization of all the generalized derivations $\D$ on $TM$ which endow the Lie bialgebroid $(TM, T^*M)$ associated to a Poisson structure with a Lie-Nijenhuis structure. It turns out that they are all of form $\D^{r,T} = (D^{r,T}, r, r)$, where $r: TM \to TM$ is an endomorphism such that $(\pi, r)$ is Poisson-Nijenhuis and
$$
D^{r,T}_X(Y) = [Y, r(X)]- r([X,Y]), \,\, X, \, Y \in \Gamma(TM).
$$
Another important class of examples is holomorphic Lie bialgebroids. In this case, the generalized derivation $\D$ and its dual $\D^\top$ codify the Dolbeault operators on the underlying real Lie algebroids $A$ and $A^*$. Note that our viewpoint to holomorphic Lie bialgebroids, though equivalent, differs from the one in \cite{LSX} (see Proposition \ref{prop:hol_bialgbds} for the precise relationship between the two approachs).

The following result (see \S \ref{sec:pn_grpd} for a proof) connects Lie-Nijenhuis bialgebroids and Poisson-Nijenhuis groupoids:

\begin{thm}\label{thm:intro}
 Let $(\G, \pi) \toto M$ be a source 1-connected Poisson groupoid and $(A, A^*)$ the corresponding Lie bialgebroid. There is a 1-1 correspondence between $K: T\G \to T\G$ multiplicative endomorphisms such that $(\pi, K)$ is a PN structure and generalized derivations $\D$ of degree 1 such that $(A, A^*, \D)$ is a Lie-Nijenhuis bialgebroid. 
\end{thm}


When $(M, \pi, r)$ is a PN-manifolds and $T^*M$ is integrable to a symplectic groupoid, Theorem \ref{thm:intro} recovers the integration of Poisson-Nijenhuis structures to symplectic-Nijenhuis groupoids established in \cite{StX}. In the holomorphic case, Theorem \ref{thm:intro} reproduces the integration of holomorphic Lie bialgebroids to holomorphic Poisson groupoids obtained in \cite{LSX} (see Theorem \ref{thm:holomorphic}).

It is import to observe that $\D^{r, T}$ and its dual are related to the tangent and cotangent lifts of $r$ under the correspondence between generalized derivations and linear vector-valued forms (see Theorem \ref{thm:lifts}). In this way, for a PN-manifold $(M, \pi, r)$, we obtain a characterization of the compatibility between $r$ and $\pi$ as a condition for the cotangent lift of $r$, $r^{\rm ctg}: T(T^*M) \to T(T^*M)$, being a Lie algebroid morphism. This allows us to give a simple alternative proof of the foundational result of Y. Kosmann-Schwarzbach characterizing PN structures as Lie bialgebroid structures on $(TM_r, T^*M)$, where $TM_r$ is a Lie algebroid structure on $TM$ with $r: TM \to TM$ as anchor and the bracket $[X,Y]_r = [r(X),Y]+[X, r(Y)]-r([X,Y])$ (see Proposition \ref{prop:alt_proof_kos}). 

We would like to mention that generalized derivations play a key role in the study of many other geometric structures. In particular, they are being explored in some work in progress to study generalizations of PN geometry to Dirac structures \cite{BDN} and holomorphic multiplicative structures on Lie groupoids \cite{Dru}.


The paper is organized as follows. In Section 2, we present the theory of generalized derivations and their duality recalling the correspondence with linear vector-valued forms. We also study PN structures under this perspective. In Section 3, we show how the classical theory of complete lifts of vector-valued forms can be put into the formalism of generalized derivations. In Section 4, we study properties of Lie-Nijenhuis bialgebroids showing how they induce PN structures on the base $M$ and produce an hierarchy of Lie-Nijenhuis bialgebroids. We also prove Theorem \ref{thm:intro} in this section. Finally, in Section 5 we apply our main result to holomorphic Lie bialgebroids.

\begin{rem}\em
\noindent

\begin{itemize}
 \item[(1)]For the sake of clarity, we note that the terminology of ``Lie-Nijenhuis bialgebroids'' used to name our main object of study is not motivated by the notion of Nijenhuis operators on Courant algebroids as studied by many authors (see \cite{Kos2} and references therein). They are related, though, in the case $l^2 = \lambda \, \mathrm{id}_A$, for some $\lambda \in \R$. In this case, there is an associated Nijenhuis operator on the double $A \oplus A^*$ (see Remark \ref{rem:courant}). In this direction, it is an interesting question to understand how to characterize Lie-Nijenhuis bialgebroids by means of some geometric structure on their doubles. We intend to adress this question elsewhere.
\item[(2)] The problem of characterizing PN groupoids infinitesimally was also treated in the paper of A. Das \cite{Das}. His main result characterizes PN groupoids infinitesimally as a Lie-bialgebroid $(A, A^*)$ endowed with a linear endomorphism $K_A: TA \to TA$ such that $(K_A, \pi_A)$ is also PN, where $\pi_A$ is the linear Poisson-structure corresponding to the Lie algebroid structure on $A^*$. To compare our viewpoints, note that, for the Poisson-groupoid $(A, \pi_A)$ (i.e. $A$ seen as a Lie groupoid with the multiplication given by fiberwise addition) his result does not give any new information whereas Theorem \ref{thm:intro} characterizes linear PN structures $(\pi_A,K_A)$ in terms of generalized derivations. So, in particular, Theorem \ref{thm:intro} recovers the main result of \cite{Das}.
\end{itemize}
\end{rem}

\subsection*{Acknowledments}
The author would like to thank Henrique Bursztyn for helpful conversations and Janusz Grabowski for the reference \cite{GU3} on cotangent lifts of vector valued forms.

\pagebreak

\section{Generalized derivations}
\subsection{Preliminaries}
\medskip \subsubsection{Definition and main properties}
Let $q:E \to M$ be a vector bundle over a smooth manifold. A derivation $\Delta: \Gamma(E) \to \Gamma(E)$ on $E$ is an $\R$-linear operator satisfying the Leibniz equation
\begin{equation}\label{eq:Delta_Leibniz}
\Delta(fu) = f \Delta(u) + (\Lie_{\sharp(\Delta)}f) u,
\end{equation}
where $\sharp(\Delta) \in \frakx(M)$ is called the \textit{symbol} of $\Delta$. The space of derivations on $E$ will be denoted by $\mathrm{Der}(E)$. In the following, we shall present a generalization of derivations on $E$ and study its properties.
\begin{dfn}\em
 A \textit{generalized derivation of degree $k$} on $E$ is a triple $\D=(D,l,r)$, where $D:\Gamma(E) \to \Omega^k(M, E)$ is a $\R$-linear map, $l \in \Omega^{k-1}(M,\mathrm{End}(E))$ and $r \in \Omega^k(M, TM)$ satisfying the following Leibniz-type equation: for $u \in \Gamma(E), \, f \in C^\infty(M)$,
\begin{align}\label{eq:D_Leibniz}
 D(fu) & = fD(u) + df \wedge l(u) - r^*(df) \otimes u.
\end{align}
Here $r^*:T^*M \to  \wedge^k T^*M$ is the map dual to $r$. The pair $(l,r)$ is called the \textit{symbol} of $\D$ and we will denote it by $\sharp(\D)$.
\end{dfn}

We will denote by $\mathrm{GDer}^k(E)$ the space of generalized derivations of degree $k$ on $E$. The equation \eqref{eq:D_Leibniz} will be refered to simply as the Leibniz equation for $\D$.

\begin{lem}\label{lem:ses}
The symbol map defines a short exact sequence of vector spaces:
$$
0 \to \Omega^k(M, \mathrm{End}(E)) \to \mathrm{GDer}^k(E) \to \Omega^{k-1}(M, \mathrm{End}(E)) \oplus \Omega^k(M, TM)  \to 0.
$$
\end{lem}

\begin{proof}
The Leibniz equation for $\D$ implies that,  if $l=0, r=0$, then $D$ is tensorial. 
The surjectiveness of the symbol map is a consequence of the following construction: given a pair $(l, r)$ and a connection $\nabla: \Gamma(TM) \times \Gamma(E) \to \Gamma(E)$, define $(D^\nabla, l, r) \in \GDer^k(E)$ as follows:
\begin{equation}\label{def:D_nabla}
D^\nabla_{(X_1,\dots, X_k)}(u) =\sum_{i=1}^k (-1)^{i+1} l_{(X_1,\dots, \widehat{X_i}, \dots, X_k)}(\nabla_{X_i} u) - \nabla_{r(X_1,\dots, X_k)} u.
\end{equation}
\end{proof}

\begin{rem}\label{rem:symbol}\em
 Given a generalized derivation $\D=(D,l,r)$, it is important to note that $D$ alone does not determine the symbol $(l,r)$ as one can see by considering $(0, \mathrm{id}_{TM}, \mathrm{id}_{TM})$, $(0,0,0) \in \GDer^1(TM)$. In any case, it follows from the Leibniz equation that $(D,r)$ (resp. $(D,l)$) determines $l$ (resp. $r$).
\end{rem}

There is a natural bijection between $\GDer^0(E)$ and $\mathrm{Der}(E)$. Indeed, for $k=0$, $\D = (D, r)$, where $D: \Gamma(E) \to \Gamma(E)$ and $r \in \frakx(M)$. If we define $\Delta = - D$, then \eqref{eq:D_Leibniz} coincides with the Leibniz equation for derivations \eqref{eq:Delta_Leibniz}. Under this identification, the short exact sequence of Lemma \ref{lem:ses} for $k=0$ recovers the Atiyah sequence:
$$
0 \to  \mathrm{End}(E) \to \mathrm{Der}(E) \to \frakx(M) \to 0.
$$
%

The Leibniz equation for a  generalized derivation $\D=(D,l,r) \in \GDer^k(E)$ allows one to extend $D$ to a $\R$-linear operator $D: \Omega^\bullet(M, E) \to \Omega^{\bullet+k}(M,E)$ by:
\begin{equation}\label{dfn:D_ext}
D(\alpha \otimes u) = \alpha \wedge D(u) + (-1)^{j}(d\alpha  \wedge l(u) - (-1)^{j(k-1)} \Lie_{r}\alpha \otimes u),
\end{equation}
where $\alpha \in \Omega^j(M)$, $u \in \Gamma(E)$. Here $\Lie_r: \Omega^\bullet(M) \to \Omega^{\bullet+k}(M)$ is the Lie derivative operator corresponding to $r$ (we refer to \cite[\S 2.8]{nat} for further details).  Any element $l \in \Omega^{k-1}(M, \mathrm{End}(E))$ can also be extended to $l: \Omega^{\bullet}(M, E) \to \Omega^{\bullet+k-1}(M,E)$ by the formula
$
l(\alpha \otimes u) = \alpha \wedge l(u),
$
so that the Leibniz equation for $\D$ extends to 
$$
D(\alpha \wedge \eta) = \alpha \wedge D(\eta) + (-1)^{i+j}(d\alpha  \wedge l(\eta) - (-1)^{(i+j)(k-1)} \Lie_{r}\alpha \wedge \eta),
$$
for $\alpha \in \Omega^j(M),\, \eta \in \Omega^i(M, E)$. 

The following result (see \cite[Cor.~5.5]{BD}) shows how the graded commutator of $\mathrm{End}_\R(\Omega(M,E))$ endows $\GDer^\bullet(E)$ with a graded Lie algebra structure. 
\noindent
\begin{prop}\label{prop:graded_lie_str}
Let $\D_i=(D_i, l_i, r_i) \in \GDer^{k_i}(E)$ be generalized derivations of degree $k_i$, for $i=1,2$. The triple $(D,l,r)$ defined as: $r=[r_1,r_2]$, the Fr\"olicher-Nijenhuis bracket of $r_1$ and $r_2$, and 
\begin{align*}
D & = (D_2 \circ D_1 - (-1)^{k_1k_2} D_1\circ D_2)|_{\Gamma(E)}\\
l & = \left([D_2, l_1] - (-1)^{k_1k_2}[D_1,l_2]\right)|_{\Gamma(E)}
\end{align*}
is a generalized derivation of degree $k_1+k_2$. Moreover, the bracket $[\D_1, \D_2] = (D,l,r)$ gives $\GDer^\bullet(E)$ a graded Lie algebra structure. 
\end{prop}
Note that the map $\Der(E) \ni \Delta \mapsto - \Delta \in \GDer^0(E)$ is a Lie algebra isomorphism.

Let us give some illustrative examples.
\begin{exam}\em
An element $\D \in \mathrm{GDer}^1(E)$ with symbol $l=\mathrm{id}_E$ and $r=0$ is the same as a connection $\nabla: \Gamma(TM) \times \Gamma(E) \to \Gamma(E)$ on $E$: simply define
 $$
 \nabla_X u = D_X(u).
 $$
The extension $D: \Omega^\bullet(M, E) \to \Omega^{\bullet+1}(M,E)$ is the Koszul differential corresponding to $\nabla$. Also,  $\frac{1}{2}[\D,\D] \in \Omega^2(M, \mathrm{End}(E))$ is the curvature of $\nabla$.  
\end{exam}

\begin{exam}\label{exam:tang_D}\em
For $r \in \Omega^k(M, TM)$, define $D^{r,T}: \Gamma(TM) \to \Omega^k(M, TM)$ by:
\begin{equation}
D^{r,T}(Y)= [Y, r].
\end{equation}
A straightforward calculation shows that $(D^{r,T}, \widetilde{r}, r) \in \mathrm{GDer}^k(TM)$, where $\widetilde{r} \in \Omega^{k-1}(M, \mathrm{End}(TM))$ is given by
$$
\widetilde{r}(X_1, \dots, X_{k-1}) = r(\cdot\,,\, X_1, \dots, X_{k-1}).
$$
The extension $D^{r,T}: \Omega^\bullet(M, TM) \to \Omega^{\bullet+k}(M, TM)$ is $D^{r,T}(\eta) = [r,\eta]$, where $[\cdot, \cdot]$ is the Fr\"olicher-Nijenhuis bracket. One can check that the graded Jacobi identity for the Fr\"olicher-Nijenhuis bracket implies that $r \mapsto D^{r,T}$ defines a graded Lie algebra monomorphism $\Omega^{\bullet}(M, TM) \hookrightarrow \GDer^\bullet(TM)$.
\end{exam}

An important example of generalized derivation comes from the Dolbeault operator on holomorphic vector bundles. This will be treated in \S \ref{sec:hol}.

\medskip \subsubsection{Duality.}
There is a fundamental duality construction for generalized derivations that extends the well-known relationship between derivations on $E$ and on $E^*$. We now present this construction and show that it determines a graded Lie algebra isomorphism between $\GDer^\bullet(E)$ and $\GDer^\bullet(E^*)$.

\begin{dfn}\em
Let $\D=(D,l,r) \in \mathrm{GDer}^k(E)$. Its dual $\D^\top \in \GDer^k(E^*)$ is defined as the triple $(D^\top, l^\top,r)$, where $D^\top$ is determined by the following equation on $\Omega^k(M)$: for $u \in \Gamma(E), \, \mu \in \Gamma(E^*)$,
\begin{equation}\label{eq:dual_D}
\<D^\top(\mu), u\> = d\<\mu, l(u)\> - \Lie_r\<\mu,u\> - \<\mu, D(u)\>, \,\,\,  
\end{equation}
and $l^\top$ is given by
$
\<l^\top(\mu), u\> = \<\mu, l(u)\>.
$
\end{dfn}

It will be useful to have a more detailed formula for $k=1$. In this case,  given $\D = (D,l,r) \in \GDer^1(E)$, 
\begin{equation}\label{eqn:1_duality}
\<D^\top_X(\mu), u\> = \Lie_X \<\mu, l(u)\> - \Lie_{r(X)}\<\mu, u\> - \<\mu, D_X(u)\>.
\end{equation}

\begin{exam}\em
For $k=0$, consider the derivations $\Delta$ on $E$ and $\Delta^\top$ on $E^*$ given by $\Delta = -D$, $\Delta^\top = -D^\top$. Using these derivations, equation \eqref{eq:dual_D} can be rewritten as
\begin{equation}\label{eq:dual_derivation}
\<\Delta^\top(\mu), u\> = \Lie_{\sharp(\Delta)} \<\mu, u\> - \<\mu, \Delta(u)\>.
\end{equation}
This is exactly the dualization operation for derivations on $E$.
\end{exam}

\begin{exam}\em
Let $\nabla: \Gamma(TM) \times \Gamma(E) \to \Gamma(E)$ be a connection on $E$ and consider the corresponding generalized derivation of degree 1, $D_X(u) = \nabla_X u$. A direct computation shows that $D^\top$ is the generalized derivation corresponding to the dual connection on $E^*$.
\end{exam}

\begin{exam}\em
For $r \in \Omega^k(M, TM)$ we shall denote by $D^{r,T^*} \in \mathrm{GDer}^k(T^*M)$ the generalized derivation dual to $D^{r,T}$ (see Example \ref{exam:tang_D}). When $k=1$, a straightforward computation shows that: for $\alpha \in \Omega^1(M)$, 
\begin{equation}\label{D_ctg}
D^\ct_X(\alpha) = \Lie_X(r^*\alpha) - \Lie_{r(X)}\alpha.
\end{equation}
\end{exam}

Our main result regarding dualization of generalized derivations establishes that it is a graded Lie algebra isomorphism.

\begin{thm}\label{thm:duality}
 The map \eqref{eq:dual_D} $\GDer^\bullet(E) \ni \D \mapsto \D^\top \in \GDer^\bullet(E^*)$ is a graded Lie algebra isomorphism.
\end{thm}

\begin{proof}
Let $\D_i=(D_i,l_i,r_i) \in \GDer^{k_i}(E)$, for $i=1,2$ and let us compare the components of $[\D_1^\top, \D_2^\top]$ with those of  $[\D_1, \D_2]^\top$. First, note that both $r$-components are $[r_1,r_2]$. So, it suffices to compare their $D$-components (see Remark \ref{rem:symbol}). For this, the following equations on $\Omega^{j+k}(M)$ will be necessary (they follow  from formula \eqref{dfn:D_ext} for the extension $D^\top: \Omega^j(M,E^*) \to \Omega^{j+k}(M, E^*)$ and \eqref{eq:dual_D}):
\begin{align}
\label{eq:extension_1} 
\<D^\top(\eta), v\> & = (-1)^j\left(d\<\eta, l(v)\>  - (-1)^{j(k-1)} \Lie_r\<\eta, v\>\right) - \<\eta, D(v)\>\\
\label{eq:extension_2}
\<D^\top(\mu), \gamma\> & = 
 (-1)^{j(k-1)} d\<\mu, l(\gamma)\> - \Lie_r\<\mu, \gamma\> - (-1)^{jk}\<\mu, D(\gamma)\>,
\end{align}
for $\mu \in \Gamma(E^*), v \in \Gamma(E)$, $\eta \in \Omega^j(M, E^*), \, \gamma \in \Omega^j(M, E)$. 

Now,
$$
\<[D_2^\top, D_1^\top](\mu), v\> = \<D_2^\top(D_1^\top(\mu)), v\> - (-1)^{k_1k_2}\<D_1^\top(D_2^\top(\mu)),v\> = A -(-1)^{k_1k_2} B.
$$
From \eqref{eq:extension_1}, \eqref{eq:extension_2} and  using that $d \Lie_{r_2} = (-1)^{k_2} \Lie_{r_2} d$, one finds that
\begin{align*}
 A 
&   = -(-1)^{k_1}d \Lie_{r_1} \<\mu, l_2(v)\> - (-1)^{k_2k_1}(-1)^{k_2} d\Lie_{r_2}\<\mu, l_1(v)\> + \Lie_{r_1}\<\mu,D_2(v)\>\\
& \hspace{-10pt} + (-1)^{k_1k_2}\Lie_{r_2}\<\mu, D_1(v)\>  - (-1)^{k_1k_2} d\<\mu, D_1(l_2(v))\> - (-1)^{k_2(k_1-1)}d\<\mu, l_1(D_2(v))\>   \\
& \hspace{-10pt}  + (-1)^{k_1k_2}\Lie_{r_2} \Lie_{r_1} \<\mu, v\> + (-1)^{k_1k_2}\<\mu, D_1(D_2(v))\>.
 \end{align*}
As $B$ is obtained just by permutation of the indices, one can check that the first four terms of $A$ will cancel out with the corresponding terms in $(-1)^{k_1k_2}B$. By collecting the remaining terms, one has that
\begin{align*}
\<[D_2^\top, D_1^\top](\mu), v\> & 
= d\<\mu, [D_2,l_1] - (-1)^{k_1k_2} [D_1,l_2]\>  - \Lie_{[r_1,r_2]}\<\mu,v\> \\
& \hspace{-40pt} - \<\mu, [D_2,D_1](v)\>\\
&  = \<[D_2,D_1]^\top(\mu), v\>,
\end{align*}
where we have used Proposition \ref{prop:graded_lie_str} in the last equality. This concludes the proof.


\end{proof}

\medskip \subsubsection{Linear vector-valued forms}\label{sec:lin_vec_val}
It is well-known \cite[\S~3.4]{Mac-book} that derivations on a vector bundle $E$ correspond bijectively to linear vector fields on $E$ (i.e. vector fields whose flow is by vector bundle automorphisms). In this subsection, we shall first recall how generalized derivations of degree $k$ on $E$  are in 1-1 correspondence with linear vector-valued forms on $E$, $\Omega_{lin}(E,TE)$, following \cite{BD}. It is important to point out that the Fr\"olicher-Nijenhuis bracket on $\Omega(E, TE)$ endows $\Omega_{lin}(E, TE)$ with a graded Lie algebra structure and $\GDer(E) \leftrightarrow \Omega_{lin}(E,TE)$ is a graded Lie algebra isomorphism.  

A vector-valued form $K \in \Omega^k(E, TE)$ on the total space $E$ is said to be \textit{linear} if it is $h_{\lambda}$-related to itself
, for every $\lambda \in \R_{>0}$. Here $h_\lambda: E \to E$ is the fiberwise multiplication by $\lambda$. We shall denote the space of linear vector-valued forms on $E$ by $\Omega_{lin}(E,TE)$.


To explain the relationship between linear vector-valued forms and generalized derivations, it is necessary to introduce a map $\mathcal{V}: \Omega^j(M, E) \to \Omega^j(E, TE)$ given by:
\begin{equation}\label{dfn:V_lift}
\mathcal{V}(\alpha \otimes u) = q^*\alpha \otimes u^\uparrow, \,\, \alpha \in \Omega^j(M), \,u \in \Gamma(E),
\end{equation}
where $u^\uparrow \in \frakx(E)$ is the vertical lift\footnote{The vertical lift of $u \in \Gamma(E)$ is the vector field given by
$$
u^\uparrow(e) = \frac{d}{d\epsilon}(e+ \epsilon u(x)), \,\,\,e \in E_x, \, x \in M.
$$
}. From \cite[Thm.~3.19]{BD}), it is known that there is a 1-1 correspondence between $K \in \Omega^k_{lin}(E, TE)$ and generalized derivations $\D=(D,l,r) \in \GDer^k(E)$ given by:
\begin{equation}\label{dfn:D_l_r}
\mathcal{V}(D(u)) = \Lie_{u^\uparrow}K, \,\,\, \mathcal{V}(l(u)) = K(u^\uparrow, \cdot) , \,\,\,\,\,q^*\<\beta,r\> =  \<K, q^*\beta\>
\end{equation}
where $\beta \in \Omega^1(M)$. We shall refer to $(l,r)$ as the symbol of $K$ as well.

For a linear vector-valued form $K \in \Omega^k(E, TE)$ with corresponding generalized derivation $\D \in \GDer^k(E)$, we shall denote the linear vector-valued form on $E^*$ associated to the generalized derivation $\D^\top$ \eqref{eq:dual_D} by $K^\top$.

If $U \in \frakx(E)$ is the linear vector field corresponding to $\Delta$, then $U \mapsto U^\top$ is the well-known bijection between linear vector fields on $E$ and on $E^*$ (see \cite[\S~3.4]{Mac-book}).

\begin{rem}\em
In the case $k=0$, the correspondence \eqref{dfn:D_l_r} is equivalent to the well-known correspondence between derivations and linear vector fields (see e.g. \cite[\S~3.4]{Mac-book} for more details): given a linear vector field $U \in \frakx(E)$  the formula
 $$
 [U, u^\uparrow] = \Delta(u)^\uparrow
 $$
 defines a derivation $\Delta: \Gamma(E) \to \Gamma(E)$.
By comparing with \eqref{dfn:D_l_r}, it follows that the generalized derivation of degree $0$ corresponding to $U$ is exactly
$
D = -\Delta.
$
\end{rem}

In the following, we shall establish some useful properties of linear vector-valued $k$-forms on $E$.

\begin{prop}
 There exists a vector bundle $\mathrm{Lin}^k(E)$ over $M$ whose sections are the linear vector-valued $k$-forms on $E$. Moreover, $\mathrm{Lin}^k(E)$ fits into a short exact sequence of vector bundle over $M$:
 \begin{align}\label{eq:lin_ses}
 0 \longrightarrow \wedge^k T^*M \otimes \mathrm{End}(E) & \longrightarrow \mathrm{Lin}^k(E) \longrightarrow\\
 \nonumber & \hspace{-70pt}\longrightarrow (\wedge^{k-1}T^*M \otimes \mathrm{End}(E)) \oplus (\wedge^k T^*M \otimes TM) \longrightarrow 0.
 \end{align}
\end{prop}

\begin{proof}
There is a $C^\infty(M)$-module structure on the space of linear vector-valued forms on $E$ given by the scalar multiplication of sections with $f \circ q$, for $f \in C^\infty(M)$. 
Let $(x,\xi^1, \dots, \xi^n)$ be local coordinates on $E$ and $\{u_1, \dots, u_n\}$ be a frame for $E$ such that
$
\xi^i(u_j(x)) = \delta_j^i.
$
On these coordinates, it is straightforward to check (e.g. using \cite[Prop.~4.10]{BD}) that a linear vector-valued $k$-form $K$ is given by:
\begin{align}\label{eq:local_form}
\nonumber K(x,\xi^1, \dots, \xi^n) & = r^{I_k}_j(x) \, dx_{I_k} \otimes \frac{\partial}{\partial x_j} +\xi^a \cdot D_a^{I_k, b}(x) \,dx_{I_k} \otimes \frac{\partial}{\partial \xi^b}\\
& \hspace{-40pt}  +  l^{I_{k-1},\,b}_a(x) \,d\xi^a \wedge dx_{I_{k-1}} \otimes \frac{\partial}{\partial \xi^b},
\end{align}
where $I_j = \{i_1 < \dots < i_j\} \subset \{1, \dots, n\}$, for $j = k-1$ or $k$, and 
\begin{align*}
l^{I_{k-1},\,b}_a & = \left\<\varphi^b, l(u_a)\left(\frac{\partial}{\partial x_{I_{k-1}}}\right)\right\>, \,\,\,\, r^{I_k}_j(x) = \left\<dx_j, r\left(\frac{\partial}{\partial x_{I_k}}\right)\right\>,\\
D_a^{I_k,b} & = \left\<\varphi^b, D(u_a)\left(\frac{\partial}{\partial x_{I_k}}\right)\right\>,
\end{align*}
where $\{\varphi^1, \dots, \varphi^n\}$ is the frame for $E^*$ dual to $\{u_1, \dots, u_n\}$ and $(D,l,r)$ is the generalized derivation corresponding to $K$. It now follows from Lemma \ref{lem:ses} that the module is locally free and finitely generated, hence it is the space of sections of a vector bundle. The existence of the short exact sequence \eqref{eq:lin_ses} follows from the $C^{\infty}(M)$-linearity of the s.e.s of Lemma \ref{lem:ses}. 
\end{proof}

Note that $\mathrm{Lin}^0(E)$ is the Atiyah algebroid of $E$. Also, it is important to give an explicit description of the inclusion $\wedge^k T^*M \otimes \mathrm{End}(E) \hookrightarrow \mathrm{Lin}^k(E)$: for $\Phi \in \Omega^k(M, \mathrm{End}(E))$, the linear vector-valued $k$-form $\Phi^\uparrow \in \Omega^k(E, TE)$ is given as
\begin{equation}\label{eq:vertical_lin}
\Phi^\uparrow(U_1, \dots, U_k) = \left.\frac{d}{d\epsilon}\right|_{\epsilon=0}(e + \epsilon \,\Phi_{(Tq(U_1), \dots, Tq(U_1))}(e)).
\end{equation}
For $\mathrm{id}_E \in \Omega^0(M, \mathrm{End}(E))$, $\mathrm{id}_E^{\uparrow} = \mathcal{E} \in \mathfrak{X}(E)$ is the Euler vector field of $E$.


The following lemma will be a useful tool to study vector-valued forms on vector bundles.
\begin{lem}\label{lem:vec_on_lin}
 Let $K, L \in \Omega^k(E,TE)$ be vector-valued forms on $E$. If
 $$
 K(U_1,\dots, U_k) = L(U_1, \dots, U_k),
 $$
 for any $k$-tuple of linear vector fields $U_1, \dots, U_k \in \frakx(E)$, then $K=L$.
\end{lem}

\begin{proof}
 It suffices to proof that
$
 K(U_1, \dots, U_k) = 0 \Rightarrow K \equiv 0,
$
for any $k$-tuple $(U_1, \dots, U_k)$ of linear vector fields. This follows from writing $K$ in local coordinates and choosing appropriated linear vector to show that all its components vanish.
\end{proof}

\begin{prop}\label{prop:lin_char}
Let $K \in \Omega^k(E, TE)$. The following are all equivalent:
\begin{itemize}
 \item[(i)] $K$ is linear;
 \item[(ii)] $\Lie_{\mathcal{E}} K \equiv 0$, where $\mathcal{E} \in \mathfrak{X}(E)$ is the Euler vector field of $E$;
 \item[(iii)] For any $k$-tuple of linear vector fields $U_1, \dots, U_k$, the vector field $U=K(U_1, \dots, U_k)$ is itself linear.
\end{itemize}
In this case, the derivation $\Delta$ corresponding to $U$ is given by
\begin{equation}\label{eq:K_der}
\Delta(u) = \sum_{i=1}^k (-1)^{i+1} l_{(\sharp(U_1), \dots, \widehat{\sharp(U_i)}, \dots, \sharp(U_k))}(\Delta^i(u)) - D_{(\sharp(U_1), \dots, \sharp(U_k))}(u),
\end{equation}
where $\Delta_i$, $i=1,\dots, k$, are the derivations corresponding to $U_i$ and $\sharp(U_i)$ are their symbols.
\end{prop}

\begin{proof}
It can be checked that $K \in \Omega^k(E,TE)$ satisfies $\Lie_\mathcal{E} K \equiv 0$ if and only if it is locally written in the form \eqref{eq:local_form} for some functions $l^{I_{k-1},\,b}_a(x), \, r^{I_k}_j(x)$ and $D_a^{I_k,b}(x)$. This is a consequence of the following facts for a function $f \in C^{\infty}(E)$: 1) $f$ is fiberwise constant if and only if $\Lie_{\mathcal{E}}f = 0$; 2) $f$ is fiberwise linear if and only if $\Lie_{\mathcal{E}}f = f$; 3) $f \equiv 0$ if and only if $\Lie_{\mathcal{E}}f = -n f$, for some positive integer $n$. So, the equivalence $(i) \Leftrightarrow (ii)$ follows from the fact that any $K$ written locally as \eqref{eq:local_form} is $h_\lambda$-related to itself. As for the equivalence $(ii) \Leftrightarrow (iii)$, one knows, from Lemma \ref{lem:vec_on_lin}, that $\Lie_{\mathcal{E}}K \equiv 0$ if and only if $(\Lie_{\mathcal{E}}K)(U_1, \dots, U_k) = 0$, for any $k$-tuple of linear vector fields $(U_1, \dots, U_k)$. But,
\begin{align*}
0&=(\Lie_{\mathcal{E}}K)(U_1, \dots, U_k) = [\mathcal{E}, K(U_1, \dots, U_k)] - \sum_{i=1}^k K(U_1, \dots, \cancel{[\mathcal{E},U_i]}^{\,\,=0}\hspace{-5pt}, \dots, U_k) \\
& = [\mathcal{E}, K(U_1, \dots, U_k)].
\end{align*}
The formula \eqref{eq:K_der} follows from \eqref{dfn:D_l_r} and
\begin{align*}
\Delta(u)^\uparrow & = [K(U_1,\dots, U_k), u^\uparrow]\\
& = -(\Lie_{u^\uparrow}K)(U_1,\dots, U_k) + \sum_{i=1}^k (-1)^{i+1}K([U_i, u^\uparrow], U_1, \dots, \widehat{U_i}, \dots, U_k).
\end{align*}

\end{proof}

Consider now the Fr\"olicher-Nijenhuis bracket $[\cdot, \cdot]$ on $\Omega(E,TE)$. The graded Jacobi identity together with Proposition \ref{prop:lin_char} implies that $[\cdot, \cdot]$ induces a graded Lie algebra structure on $\Omega_{lin}(E,TE)$. Indeed, given $K_1, K_2 \in \Omega_{lin}(E,TE)$,
$$
\Lie_{\mathcal{E}}[K_1, K_2] = [\mathcal{E}, [K_1,K_2]] = [[\mathcal{E}, K_1], K_2] - (-1)^{k_1}[K_1, [\mathcal{E},K_2]] = 0.
$$

For future reference, we state the following result from \cite[Prop.~5.4]{BD} as a Proposition. 

\begin{prop}\label{prop:graded_lie_corresp}
The correspondence between $\Omega_{lin}(E,TE)$ and $\GDer(E)$ given by \eqref{dfn:D_l_r} is a graded Lie algebra isomorphism.
\end{prop}

\subsection{Generalized 1-derivations on Lie algebroids.}\label{sec:der_on_algbds}
Let $(A, [\cdot, \cdot], \rho)$ be a Lie algebroid over $M$. In this paper, we will be mostly interested in generalized derivations $\D \in \GDer^1(A)$ of degree $1$. We say that $\D$ is compatible with the Lie algebroid structure if
\begin{align}
\tag{IM1}\label{IM1} D_X([a,b]) &= [a, D_X(b)] - [b, D_X(a)] + D_{[\rho(b), X]} (a) - D_{[\rho(a), X]}(b)\\
\tag{IM2}\label{IM2} l([a,b]) & = [a, l(b)] - D_{\rho(b)}(a)\\
\tag{IM3} \label{IM3}r([\rho(a), X]) & = [\rho(a),r(X)] - \rho(D_{X}(a))\\
\tag{IM4}\label{IM4} r \circ \rho & = \rho \circ l.
\end{align}
This set of equations is usually refered as \textit{IM-equations} and $\D$ is called an IM $(1,1)$-tensor (where IM stand for infinitesimally multiplicative). Under the correspondence between generalized derivations and linear vector-valued forms, the IM $(1,1)$-tensors correspond to Lie algebroid morphisms $K: TA \to TA$, where $TA \to TM$ is the tangent Lie algebroid of $A \to M$ (see \cite{BD}).

It is interesting to note that \eqref{IM3} can be rewritten as
\begin{equation}\label{IM3_alt}
 \tag{IM3'} D^{r,T}_X(\rho(a)) = \rho(D_X(a)),
\end{equation}
where $D^{r,T}$ is the generalized derivation corresponding to $r$ (see Example \ref{exam:tang_D}).

Also, one can re-interpret \eqref{IM2} as follows: define a bracket ($\R$-bilinear operation) on $\Gamma(A)$ by
\begin{equation}\label{dfn:pseudo_lie}
[a,b]_\D = [l(a),b] + D_{\rho(b)}(a)
\end{equation}
and note that \eqref{IM2} can be rewritten as
\begin{equation}\label{IM2_alt}
 \tag{IM2'} [a,b]_\D  = [l(a),b] + [a,l(b)] - l([a,b]).
\end{equation}

One can check that $[\cdot,\cdot]_\D$ defines a pseudo-Lie algebroid structure on $A$ with anchors $\rho_l = \rho \circ l$ and $\rho_r = r \circ \rho$ (see \cite[Dfn.~2.1]{GU2}). From \eqref{IM2_alt}, it is clear that \eqref{IM2} implies that $[\cdot, \cdot]_\D$ is skew-symmetric. 

\begin{lem}\label{lem:Lie_D}
The bracket $[\cdot, \cdot]_\D$ is skew-symmetric if and only if 
\begin{equation*}
 K^\top \circ \pi_A^\sharp = \pi_A^\sharp \circ (K^\top)^*.
\end{equation*}
where $\pi_A \in \frakx^2(A^*)$ is the linear Poisson structure on $A^*$ and $K^\top: T(A^*) \to T(A^*)$ is the linear endomorphism corresponding to $\D^\top \in \GDer^1(A^*)$. In this case, $[\cdot, \cdot]_\D$ is the pre-Lie algebroid structure corresponding to the linear bivector field $\pi_K \in \frakx^2(A^*)$ given by
$$
\pi_K^\sharp = K^\top \circ \pi_{A}^\sharp.
$$ 
\end{lem}

We shall give a proof of Lemma \ref{lem:Lie_D} in a more general setting in \S \ref{sec:pn_grpd} (see Remark \ref{rem:lieD_from_grpd}).

Let us give some examples.
\begin{exam}\em
 Given a map $\theta: TM \to A$, define $D: \Gamma(TM) \to \Omega^1(M, A)$, $l: A \to A$ and $r: TM \to TM$ by
 $$
 D^{\theta}_X(a) = [a, \theta(X)] - \theta([\rho(a), X]), \, l = \theta \circ \rho, \, r = \rho \circ \theta.
 $$
 It is straightforward to check that $(D, l, r) \in \GDer^1(A)$ is a IM (1,1)-form. When $A=TM$ with $\rho=\mathrm{id}_{TM}$ and $[\cdot,\cdot]$ the Lie bracket of vector fields, this construction gives exactly the generalized derivation $D^{r,T}$ of Example \ref{exam:tang_D}. In particular, $\D^{r,T}$ is an IM (1,1)-tensor on $TM$ and
 \begin{equation}\label{eq:deform_bracket}
  [X,Y]_{\D^{r,T}} = [r(X),Y]+[X, r(Y)] - r([X,Y]).
 \end{equation}
 \end{exam}

\begin{exam}\em
Let $\nabla: \Gamma(TM) \times \Gamma(A) \to \Gamma(A)$ be a connection and consider $D_X(a) = \nabla_X a$, $l=\mathrm{id}_A$ and $r=0$. One can check that $\D$ is an IM $(1,1)$-tensor if and only if  $\rho \equiv 0$ and 
$
\nabla_X[a,b] = [\nabla_Xa, b] + [a, \nabla_Xb].
$
This is equivalent to $A$ being a Lie algebra bundle (see \cite[Thm.~6.4.5]{Mac-book}).
\end{exam}

\subsection{Poisson-Nijenhuis structures}
Let $\pi \in \frakx^2(M)$ be a Poisson structure on $M$ and $r: TM \to TM$ a vector-valued 1-form. We say that $\pi$ and $r$ are \textit{compatible} if $r \circ \pi^\sharp = \pi^\sharp \circ r^*$ and 
\begin{equation}\label{eq:YM_concomitant}
C_\pi^r(\alpha, \beta) := [\alpha,\beta]_{\pi_r} - ([r^*\alpha,\beta]_\pi + [\alpha,r^*\beta]_\pi - r^*([\alpha,\beta]_\pi)) = 0,
\end{equation}
where $[\cdot, \cdot]_B$ is the Lie bracket on $\Gamma(T^*M)$ associated to a bivector field $B \in \frakx^2(M)$ given by 
$$
[\alpha,\beta]_B = \Lie_{B^\sharp(\alpha)}\beta - i_{B^\sharp(\beta)}d\alpha
$$
and $\pi_r \in \frakx^2(M)$ is the bivector field defined by $\pi_r^\sharp = r \circ \pi^\sharp$.

\begin{dfn}\em
A \textit{Poisson-Nijenhuis structure} is a compatible pair $(\pi, r)$ such that the Nijenhuis torsion of $r$, $N_r$, vanishes.
\end{dfn}

For a compatible pair $(\pi, r)$, the vanishing of the Nijenhuis torsion of $r$ is equivalent to $\pi_r$ being a Poisson structure (i.e. $[\pi_r, \pi_r] = 0$). The expression $C_\pi^r$ is called the concomitant of $\pi$ and $r$. It is important to note that $\<C_\pi^r(\alpha, \beta), X\> = \<\beta, R_{\pi}^r(X,\alpha)\>$, where 
$$
R_\pi^r(\alpha, X) = \pi^\sharp(\Lie_Xr^*(\alpha) - \Lie_{r(X)}\alpha) - [\pi^\sharp(\alpha),r](X)
$$
This expression $R_\pi^r$ was the original tensor used to express the compatibility between $\pi$ and $r$ in \cite{MaMo}. We will called it the Magri-Morosi concomitant. Note that
\begin{equation}\label{eq:mm_gen_der}
 R_\pi^r(\alpha, X) = \pi^\sharp(D^{r,T^*}_X(\alpha)) - D_X^{r,T}(\pi^\sharp(\alpha)),
\end{equation}
so that $R_\pi^r = 0$ is equivalent to \eqref{IM3_alt} for $\D^{r,T^*}$ on the cotangent Lie algebroid $(T^*M, [\cdot,\cdot]_\pi, \pi^\sharp)$. Similarly, the bracket $[\cdot,\cdot]_{\D^{r,T^*}}$ is exactly $[\cdot, \cdot]_{\pi_r}$ and the vanishing of the concomitant $C_\pi^r$ is exactly \eqref{IM2_alt}. In fact, we have the following result (see \cite[Prop.~6.7]{BD}) connecting Poisson-Nijenhuis structures with $\D^{r, T^*}$.

\begin{prop}\label{prop:pn}\cite{BD}
A Poisson structure $\pi \in \frakx^2(M)$ and the vector-valued form $r: TM \to TM$ are compatible if and only if $\D^{r, T^*}$ is an IM $(1,1)$ tensor on the Lie algebroid $(T^*M, [\cdot,\cdot]_\pi, \pi^\sharp)$.
\end{prop}

\begin{rem}\em
In \cite{BD}, the operator $D^{r, T^*}$ has a different formula from \eqref{D_ctg}, but it is a straighforward calculation to show the formulas agree. 
\end{rem}

%
%

\section{Complete lifts of vector-valued forms}
In this section, we use the framework of \S 2 to revisit the classical theory of complete lifts for vector-valued forms. We shall prove that complete lifts to the tangent and cotangent bundle define linear vector valued forms dual to each other; their corresponding generalized derivations are $\D^{r, T}$ and $\D^{r, T^*}$, respectively, for $r \in \Omega^k(M, TM)$. It will become clear that generalized derivations provide a conceptually simple tool to study properties of such lifts (e.g. local formulas and special cases - such as almost complex structures).

\subsection{Definitions}
Let us briefly recall the definitions of tangent and cotangent lift (also known as complete lifts in the literature) for vector-valued forms. Our references here are \cite{GU, Mac-book, YI}.

\medskip \subsubsection{Tangent lift}
To define the complete lift to the tangent bundle of a vector-valued form, it will be necessary to  first recall the procedure of lifting vector fields and differential forms on $M$ to its tangent bundle $q_M: TM \to M$. 

\paragraph{\bf Vector fields.}
For a vector field $X \in \frakx(M)$, its complete lift to $TM$ (or, tangent lift) is the vector field $X^{\mathrm{tg}} \in \frakx(TM)$ whose flow is the derivative of the flow of $X$. We list below some alternative ways to characterize the tangent lift: 
\begin{itemize}
 \item The vector field $X^{\mathrm{tg}}$ is linear and the derivation $\Delta^X: \Gamma(TM) \to \Gamma(TM)$ corresponding to it is
 $
 \Delta^{X} (Y) = [X,Y]
 $
 (see \cite[Example~3.4.8]{Mac-book}).
 \item By seeing $X$ as a map $X: M \to TM$, its tangent lift $X^{\mathrm{tg}}: TM \to T(TM)$ satisfies
 $
 X^{\mathrm{tg}} = J \circ TX,
 $
 where $J: T(TM) \to T(TM)$ is the canonical involution of the double tangent bundle
 \begin{equation}\label{dfn:can_involution}
 J(\left.\frac{\partial^2}{\partial s \, \partial t}\right|_{s,t=0} m(s,t)) = \left.\frac{\partial^2}{\partial t \, \partial s}\right|_{s,t=0} m(s,t),
 \end{equation}
 for every map $m: (-\epsilon, \epsilon) \times (-\delta, \delta) \to M$ (see \cite[Thm.~9.6.6]{Mac-book}).
\end{itemize}

\paragraph{\bf Differential forms.}
For $\alpha \in \Omega^k(M)$, its tangent lift $\alpha^{\mathrm{tg}} \in \Omega^k(TM)$ is defined as follows: consider the map
$$
\mathfrak{F}_\alpha: \prod_{q_M}^k TM := \underbrace{TM \times_M \dots \times_M TM}_{k-\mathrm{times}} \to \R, \,\, \mathfrak{F}_\alpha(X_1, \dots, X_k) = \alpha(X_1, \dots, X_k)
$$
and define, for $U_1, \dots, U_k \in T_X(TM)$,
$$
\alpha^{\mathrm{tg}}(U_1, \dots, U_k) := T\mathfrak{F_\alpha}(J(U_1), \dots, J(U_k)),
$$
where $J$ is the canonical involution \eqref{dfn:can_involution}. Note that we are using the identification  
$
T(\prod_{q_M}^k TM) \cong \prod_{Tq_M}^k T (TM)
$
and, also, that
$
Tq_M(J(U_1)) = \dots = Tq_M(J(U_k)) = X.
$

\begin{lem}\label{lem:form_lift}
 For $\alpha \in \Omega^k(M)$, one has that
 \begin{itemize}
 \item[(a)] $i_{X^\uparrow} \alpha^{\mathrm{tg}} = q_M^*(i_X\alpha)$;
 \item[(b)] $\alpha^{\mathrm{tg}}(X_1^{\mathrm{tg}}, \dots, X_k^{\mathrm{tg}}) = \ell_{df_\alpha}$;
 \end{itemize}
 where $X, X_1, \dots, X_k \in \frakx(M)$ are vector fields, $f_\alpha = \alpha(X_1,\dots, X_k) \in C^\infty(M)$ and $\ell_{df_\alpha} \in C^\infty(TM)$ is the fiberwise linear function corresponding to $df_\alpha$. 
\end{lem}

\begin{proof}
 For $Y \in T_xM$, note that $J(X^\uparrow(Y)) = T0(Y) + \left.\frac{d}{d\epsilon}\right|_{\epsilon=0}(\epsilon X(y))$, where $T0(Y)$ is the derivative of the zero section $0:M \to TM$ evaluated on $Y$. Hence, for $U_1, \dots, U_k \in T_Y(TM)$, one can write
 \begin{align*}
 (J(X^\uparrow(Y)), J(U_1), \dots, J(U_{k-1})) & = (T0(Y), J(U_1), \dots, J(U_k))\\
 & \hspace{-50pt}+ (\left.\frac{d}{d\epsilon}\right|_{\epsilon=0}(\epsilon X(y)), 0_{Tq_M(U_1)}, \dots, 0_{Tq_M(U_{k-1})}),
 \end{align*}
 where the sum is relative to the tangent bundle $\prod_{Tq_M}^kT(TM) \to \prod_{q_M}^k TM$. By definition,
 \begin{align*}
  \alpha^{\mathrm{tg}}(X^\uparrow(Y), U_1, \dots, U_{k-1}) &  = T\mathfrak{F}_\alpha(J(X^\uparrow(Y)), J(U_1), \dots, J(U_{k-1}))\\
  & = \bcancel{T\mathfrak{F}_\alpha(T0(Y), J(U_1), \dots, J(U_k))}_{\,\,\,=0}\\
  & \hspace{-50pt} + T\mathfrak{F}_\alpha(\left.\frac{d}{d\epsilon}\right|_{\epsilon=0}(\epsilon X(y)), 0_{Tq_M(U_1)}, \dots, 0_{Tq_M(U_{k-1})})\\
  & = \left.\frac{d}{d\epsilon}\right|_{\epsilon=0}\alpha(\epsilon X(y), Tq_M(U_1), \dots, Tq_M(U_{k-1}))\\
  & = \alpha(X(y), Tq_M(U_1), \dots, Tq_M(U_{k-1})).
   \end{align*}
   This proves (a). As for (b), note that $J(X^{\mathrm{tg}}(Y)) = TX(Y)$, since $J^2 = \mathrm{id}$. Hence,
   $$
   \alpha^{\mathrm{tg}}(X_1^{\mathrm{tg}}(Y), \dots, X_k^{\mathrm{tg}}(Y))  = T\mathfrak{F}_\alpha(TX_1(Y), \dots, TX_k(Y))
     = Tf_\alpha(Y)
     = \ell_{df_\alpha}(Y).
   $$
\end{proof}

\begin{rem}\em
 Our definition of the tangent lift for differential forms agrees with the ones in the literature \cite{BCO, GU, YI}. Indeed, it suffices to compare their local descriptions (see e.g. \cite[Eq.~2.1]{GU}). By considering local coordinates $(x, \dot{x})$ in $TM$ and writing $\alpha(x) = \alpha^{I_k}(x) \,dx_{I_k}$, it follows from Lemma \ref{lem:form_lift}, by considering the tangent and vertical lift of $\partial/\partial x_i$, that
 $$
 \alpha^{\mathrm{tg}}(x,\dot{x}) = \frac{\partial \alpha^{I_k}}{\partial x_j}(x)\,\dot{x}_j \, dx_{I_k} + \sum_{j=1}^k \alpha^{I_k}(x)\, dx_{i_1} \wedge \dots \wedge d\dot{x}_{i_j} \wedge \dots \wedge dx_{i_k}.
 $$
\end{rem}

It will be necessary to recall two properties of $\alpha^{\rm tg}$ which will be needed later:
\begin{align}
 (d\alpha)^{\rm tg} & = d(\alpha^{\rm tg})\\
\label{eq:euler_form}  \Lie_{\mathcal{E}}\alpha^{\rm tg} & = \alpha^{\rm tg}.
\end{align}
Both properties are implied by the local formula of $\alpha^{\rm tg}$. 


\paragraph{\bf Vector-valued forms.} We follow \cite{YI} to define the complete lift to the tangent bundle $r^\mathrm{tg} \in \Omega^k(TM, T(TM))$ of vector-valued forms $r \in \Omega^k(M, TM)$. First, for $r=\alpha \in \Omega^k(M)$ and $X \in \frakx(M)$, we define
\begin{equation}\label{eq:vec_val_lift}
(\alpha\otimes X)^{\mathrm{tg}} = q_M^*\alpha \otimes X^{\mathrm{tg}} + \alpha^{\mathrm{tg}} \otimes X^\uparrow
\end{equation}
and extend to general vector-valued forms by linearity.

\begin{rem}\em
To check that \eqref{eq:vec_val_lift} is well-defined, one uses the following properties of tangent lifts of vector-fields and forms: for $f \in C^{\infty}(M)$,
\begin{align*}
(fX)^{\mathrm{tg}} = (q_M^*f) X^{\mathrm{tg}} + \ell_{df} X^\uparrow; \,\,\, (f \alpha)^{\mathrm{tg}} = (q_M^*f) \alpha^{\mathrm{tg}} + \ell_{df}\, q_M^*\alpha,
\end{align*}
where $\ell_{df} \in C^\infty(TM)$ is the fiberwise linear function corresponding to $df \in \Gamma(T^*M)$ (see \cite[Eq.~(39) in \S 9.6]{Mac-book} and \cite[Lemma~3.1(ii)]{BCO}, respectively). From this, one has that
\begin{align*}
 (f \alpha \otimes X)^{\mathrm{tg}} & = ((f\alpha) \otimes X)^{\mathrm{tg}}  = (\alpha \otimes (fX))^{\mathrm{tg}}\\
 & = (q_M^*f)(\alpha \otimes X)^{\mathrm{tg}} + \ell_{df} (q_M^*\alpha \otimes X^\uparrow),
 \end{align*}
 for $f \in C^\infty(M)$.
\end{rem}

\medskip \subsubsection{Cotangent lift}
We follow \cite{GU, GU3} to define cotangent lifts for vector-valued forms. It will depend on a general construction for manifolds $N$ endowed with a bivector field $\pi \in \Gamma(\wedge^2 TN)$. Let $R_\pi: \Omega^k(N) \to \Omega^{k-1}(N, TN)$ be the map given as follows:
for $\alpha \in \Omega^k(N)$, $k \geq 1$,  consider $\widetilde{\alpha}:\wedge^{k-1} TN \to T^*N$ determined by 
$$
(X_1, \dots, X_{k-1}) \mapsto \alpha(\cdot\,,\, X_1, \dots, X_{k-1}),
$$
and define 
$$
R_\pi|_{C^\infty(M)} =0, \,\, 
R_\pi(\alpha) = \pi^\sharp \circ \widetilde{\alpha},
$$
where $\pi^\sharp: T^*M \to TM$ is the contraction map. Note 
\begin{itemize}
\item[(i)] $R_\pi|_{\Omega^1(N)} = \pi^\sharp$
\item[(ii)] $
R_\pi(\alpha \wedge \beta) = R_{\pi}(\alpha)\wedge \beta + (-1)^{k} \alpha  \wedge R_\pi(\beta),
$
for $\beta \in \Omega(N)$.
\item[(iii)]  For $\mu_1, \dots, \mu_k \in \Omega^1(N)$,
\begin{equation}\label{eq:R_on_product}
R_\pi(\mu_1 \wedge \dots \wedge \mu_k) = \sum_{i=1}^k (-1)^{i+1} \mu_1 \wedge \dots \wedge \widehat{\mu_i} \dots \wedge \mu_k \otimes \pi^\sharp(\mu_i).
\end{equation}
\end{itemize}
We refer to \cite[\S~3.1]{Mic} for further details.

We shall now focus on $N=T^*M$ with the Poisson structure $\pi_{\rm can}: = \omega_{\rm can}^{-1}$, where $\omega_{\rm can}$ is the canonical symplectic structure. Recall that, for $X \in \frakx(M)$, $\alpha \in \Omega^1(M)$,
$$
\pi_{\rm can}^\sharp(d\ell_X) = X^{\rm ctg}, \,\, \pi^\sharp_{\rm can}(p_M^*\alpha) = - \alpha^{\uparrow},
$$
where $p_M: T^*M \to M$ is the cotangent bundle projection, $\ell_X \in C^{\infty}(T^*M)$ is the fiberwise linear map corresponding to $X$ and $X^{\rm ctg}$ is the cotangent lift of $X$,  the linear vector field whose corresponding derivation is $\Lie_X: \Gamma(T^*M) \to \Gamma(T^*M)$, the Lie derivative along $X$.

Now, for $r \in \Omega^k(M, TM)$, its cotangent lift $r^{\mathrm ctg} \in \Omega^k(T^*M, T(T^*M))$ is defined by
\begin{align*}
(\alpha \otimes X)^{\rm ctg} & = R_{\pi_{\rm can}}(d(\ell_X  \,p_M^*\alpha))\\
& = p_M^*\alpha \otimes X^{\rm ctg} - d\ell_X \wedge R_{\pi_{\rm can}}(p^*_M \alpha) + \ell_X R_{\pi_{\rm can}}(p_M^*d\alpha)
\end{align*}
in the case $r = \alpha \otimes X$. Note that it is well-defined and extends to the general case by $\R$-linearity.

It will be important to note that, for $\beta \in \Omega(M)$, 
\begin{equation}\label{eq:R_pull_back}
R_{\pi_{\rm can}}(p_M^*\beta) = - \mathcal{V}(\widetilde{\beta}),
\end{equation}
for the map $\mathcal{V}$ defined by \eqref{dfn:V_lift}. This can be directly checked using \eqref{eq:R_on_product} on local coordinates. So, we have the following expression for the cotangent lift of $r = \alpha \otimes X$:
\begin{equation}\label{eq:ctg_lift}
 r^{\rm ctg } = p_M^*\alpha \otimes X^{\rm ctg} + d\ell_X \wedge \mathcal{V}(\widetilde{\alpha}) - \ell_X \,\mathcal{V}(\widetilde{d\alpha})
\end{equation}

%

\subsection{Presentation as linear vector-valued forms}
We can now state our main result regarding tangent and cotangent lifts.
\begin{thm}\label{thm:lifts}
Given $r \in \Omega^k(M, TM)$, its tangent lift $r^{\rm tg}$ (resp. cotangent lift $r^{\rm ctg}$) is the linear vector-valued form on $TM$ (resp. $T^*M$) associated to $D^{r,T}$ (resp. $D^{r, T^*}$). In particular, $r^{\rm tg} = (r^{\rm ctg})^\top$.
\end{thm}

\begin{proof}
Let us assume that $r = \alpha \otimes X$. The general case will follow by $\R$-linearity. In the following, $\mathcal{E}_T$ (resp. $\mathcal{E}_{T^*}$) will denote the Euler vector field on $TM$ (resp. $T^*M$). Also, we will use the following property of the Fr\"olicher-Nijenhuis bracket on an arbitrary manifold $N$: for $Y \in \mathfrak{X}(N)$, $\beta \in \Omega(N)$ and $K \in \Omega(N, TN)$
$$
[Y, \beta \wedge K] = \beta \wedge [Y,K] + (\Lie_Y\beta) \wedge K
$$
(see \cite[\S 8.7]{nat}). 
\paragraph{\bf Tangent lift.}
To prove that $r^{\rm tg}$ is linear, we show that $[\mathcal{E}_T, r^{\rm tg}]= 0$ and use Proposition \ref{prop:lin_char}. Now,
\begin{align*}
[\mathcal{E}_T, r^{\rm tg}] & = [\mathcal{E}_T, q_M^*\alpha\otimes X^{\rm tg} + \alpha^{\rm tg}\otimes X^\uparrow]\\
& \hspace{-20pt} = q_M^*\alpha \otimes \cancel{[\mathcal{E}_T, X^{\rm tg}]}^{=0} + \cancel{\Lie_{\mathcal{E}_T}q_M^*\alpha}^{=0} \otimes X^{\rm tg} + \alpha^{\rm tg} \otimes [\mathcal{E}_T, X^\uparrow] + \Lie_{\mathcal{E}_T}\alpha^{\rm tg} \otimes X^\uparrow\\
& \hspace{-20pt}= -\alpha^{\rm tg} \otimes X^{\uparrow} + \alpha^{\rm tg} \otimes X^\uparrow = 0.
\end{align*}
Here, we have used \eqref{eq:euler_form} and that $[\mathcal{E}_T, X^\uparrow]=-X^\uparrow$.  Let us now check that $\D^{r,T}$ is the generalized derivation corresponding to $r^{\rm tg}$. For this, we shall use \eqref{dfn:D_l_r}. First,
\begin{align*}
[Y^\uparrow, r^{\rm tg}] & = q_M^*\alpha \otimes [Y^\uparrow, X^{\rm tg}] + \cancel{\Lie_{Y^\uparrow}q_M^*\alpha}^{=0} \otimes X^{\rm tg} \\
& \hspace{-20pt} + \alpha^{\rm tg}\otimes \cancel{[Y^\uparrow, X^\uparrow]}^{=0} + \Lie_{Y^\uparrow}\alpha^{\rm tg} \otimes X^\uparrow\\
& = q_M^*\alpha \otimes [Y,X]^\uparrow + q_M^*(\Lie_X\alpha) \otimes X^\uparrow\\
& = \mathcal{V}([Y, r]) = \mathcal{V}(D^{r,T}(Y)).
\end{align*}
It is now a straightforward computation to check that the symbol of $r^{\rm tg}$ is $(\widetilde{r}, r)$. This concludes the tangent lift part of the proof.

\paragraph{\bf Cotangent lift.}
Similarly, let us prove that $r^{\rm ctg}$ is linear:
 \begin{align*}
  [\mathcal{E}_{T^*}, r^{\rm ctg}] & = q_M^*\alpha \otimes \cancel{[\mathcal{E}_{T^*}, X^{\rm ctg}]}^{=0} + \cancel{\Lie_{\mathcal{E}_{T^*}}q_M^*\alpha}^{=0} \otimes X^{\rm ctg}  + d\ell_X \wedge [\mathcal{E}_{T^*}, \mathcal{V}(\widetilde{\alpha})]\\ 
  & \hspace{-20pt} + \Lie_{\mathcal{E}_{T^*}}(d\ell_X) \wedge \mathcal{V}(\widetilde{\alpha}) - \ell_X [\mathcal{E}_{T^*}, \mathcal{V}(\widetilde{d\alpha})] - \Lie_{\mathcal{E}_{T^*}}(\ell_X) \mathcal{V}(\widetilde{d\alpha})\\
  &  = - d\ell_X \wedge \mathcal{V}(\widetilde{\alpha}) + d\ell_X \wedge \mathcal{V}(\widetilde{\alpha}) + \ell_X  \mathcal{V}(\widetilde{d\alpha}) - \ell_X \mathcal{V}(\widetilde{d\alpha})
  = 0,
 \end{align*}
where we have used that $\Lie_{\mathcal{E}_{T^*}} \ell_X = \ell_X$ and $[\mathcal{E}_{T^*}, \mathcal{V}(\gamma)]= -\gamma$, for $\gamma \in \Omega(M,T^*M)$. So, by Proposition \ref{prop:lin_char}, we conclude that $r^{\rm ctg}$ is a linear vector-valued form.

Let us prove that the generalized derivation associated to $r^{\rm ctg}$ is $D^{r, T^*}$. Let $\mu \in \Omega^1(M)$. From \eqref{eq:dual_D}, one has that
$$
\<D^{r,T^*}(\mu), Y\> = d\<\mu, X\> \wedge i_Y \alpha - \<\mu, X\> \,i_Y d\alpha - \<\Lie_X \mu, Y\>\, \alpha.
$$
Now, 
\begin{align*}
[\mu^\uparrow, r^{\rm ctg}] & = p_M^*\alpha \otimes [\mu^\uparrow, X^{\rm ctg}] + \cancel{\Lie_{\mu^\uparrow} p_M^*\alpha}^{=0} \otimes X^{\rm ctg} + d\ell_X \wedge \cancel{[\mu^\uparrow, \mathcal{V}(\widetilde{\alpha})] }^{=0} \\
& \hspace{-30pt} + \Lie_{\mu^\uparrow}d\ell_X  \wedge \mathcal{V}(\widetilde{\alpha}) - \ell_X \cancel{[\mu^\uparrow, \mathcal{V}(\widetilde{d\alpha})]}^{=0} - (\Lie_{\mu^\uparrow}\ell_X) \mathcal{V}(\widetilde{\alpha})\\
 & = - p_M^*\alpha \otimes (\Lie_X \mu)^\uparrow + p_M^* d\<\mu, X\> \wedge   \mathcal{V}(\widetilde{\alpha}) - p_M^*\<\mu, X\> \, \mathcal{V}(\widetilde{d\alpha})\\
 & = \mathcal{V}(\underbrace{\alpha \otimes \Lie_X \mu + d\<\mu, X\> \wedge \widetilde{\alpha} - \<\mu, X\> \widetilde{d\alpha}}_{\gamma \in \Omega^k(M, T^*M)}).
\end{align*}
It is now straightforward to check that $\<\gamma, Y\> = \<D^{r,T^*}(\mu), Y\>$, for any $Y \in \mathfrak{X}(M)$. So, $[\mu^\uparrow, r^{\rm ctg}] = \mathcal{V}(D^{r,T^*}(\mu))$ as we wanted to prove.  We leave to the reader to check that the symbol of $r^{\rm ctg}$ is $(\widetilde{r}^\top, r)$.

\end{proof}

We end this section with two corollaries. Both of them are related to Lie algebroid endomorphisms of tangent Lie algebroids $TA \to TM$. For the first, we consider $A=TM$ and, for the second, $A= T^*M$ with the Lie algebroid structure coming from a Poisson structure $\pi \in \frakx^2(M)$ on $M$.

\begin{cor}\label{cor:tangent_1_1}
Given a linear vector-valued 1-form $K \in \Omega^1(TM, T(TM))$ with symbol $(l,r)$, one has that $K: T(TM) \to T(TM)$ is a Lie algebroid endomorphism if and only if $K = r^{\rm tg}$. 
\end{cor}

\begin{proof}
Let $\D = (D, r, l) \in \GDer^1(TM)$. We claim that $\D$ is a IM (1,1) form on $TM$ if and only if $\D= \D^{r,T}$. The result will then follow from Theorem \ref{thm:lifts}. To prove the claim, note that, since $\rho = \mathrm{id}_{TM}$ in this case, equation \eqref{IM3_alt} is equivalent to $D = D^{r,T}$ and \eqref{IM4} gives that $l=r$.
\end{proof}

\begin{cor}\label{cor:pn_ctg_char}
Given a vector-valued 1-form $r \in \Omega^1(M)$, one has that $\pi$ and $r$ are compatible if and only if $r^{\rm ctg}: T(T^*M) \to T(T^*M)$ is a Lie algebroid morphism.
\end{cor}

\begin{proof}
 This is a straighforward consequence of Proposition \ref{prop:pn} and Theorem \ref{thm:lifts}.
\end{proof}

\section{Lie-Nijenhuis bialgebroids}
\subsection{Preliminaries} \label{subsec:IM_tensors}
Let $\G \toto M$ be a Lie groupoid and consider the Lie groupoid structure on its tangent and cotangent bundles, $T\G \toto TM$ and $T^*\G \toto A^*$, respectively. Here $A \to M$ is the Lie algebroid of $\G$. We shall refer the reader to \cite[\S 7]{Mac-Xu} for details and references relative to tangent and cotangent groupoids.

Following \cite{BD} we shall say that a $(q,p)$-tensor field $\tau \in \Gamma(\wedge^p T^*\G \otimes \wedge^q T\G)$ is \textit{multiplicative} if the associated function on the Whitney sum $\bigoplus^p T\G \oplus \bigoplus^q T^*\G$ is a multiplicative function. This notion recovers all the existing definitions of multiplicative tensors in the literature (e.g. Poisson groupoids \cite{Mac-Xu, We88}, symplectic groupoids \cite{Kar,We87}, (1,p) tensors \cite{LMX}). 

Let us recall the infinitesimal description of multiplicative tensors. For this, consider the $\R$-linear map $\mathcal{T}: \Gamma(\wedge^p T^*M \otimes \wedge^q A) \to \Gamma(\wedge^p T^*\G \otimes \wedge^q T\G)$ given by
$$
\mathcal{T}(\alpha \otimes (a_1 \wedge \dots \wedge a_q)) = \t^*\alpha \otimes (\overrightarrow{a_1}\wedge \dots \wedge \overrightarrow{a_q}), \,\, \alpha \in \Omega^p(M), \,\, a_i \in \Gamma(A),
$$
where $\t: \G \to M$ is the target map of $\G$ and, for $a \in \Gamma(A)$,  $\overrightarrow{a} \in \frakx(\G)$ is the corresponding right-invariant vector field. In the case $\G$ is source-connected, a $(q,p)$-tensor field $\tau$ on $\G$ is multiplicative if and only if:
\begin{itemize}
 \item[(i)] $\tau(X_1,\dots, X_p, \mu_1, \dots, \mu_q) = 0$, for $X_i \in T_x M \subset TG, \, \mu_j \in A^*_x \subset T^*\G$, $x \in M$;
\item[(ii)] there exists $\mathfrak{D}: \Gamma(A) \to \Gamma(\wedge^p T^*M \otimes \wedge^q A)$, $\mathfrak{l}: A \to \Gamma(\wedge^{p-1} T^*M \otimes \wedge^{q} A)$ and $\mathfrak{r}: T^*M \to \Gamma(\wedge^p T^*M \otimes \wedge^{q-1} A)$ such that
\begin{equation}\label{dfn:Dlr_mult}
\Lie_{\overrightarrow{a}}\tau = \mathcal{T}(\mathfrak{D}(a)), \,i_{\overrightarrow{a}}\tau = \mathcal{T}(\mathfrak{l}(a)), \, \, i_{\t^*\alpha} \tau = \mathcal{T}(\mathfrak{r}(\alpha)).
\end{equation}
\end{itemize}
Moreover, if $\tau$ is multiplicative, then $\tau \equiv 0$ if and only if $D=0,\,l=0, \,r=0$. We refer to \cite{BD} for a proof (see Theorems 3.11 and 3.19 therein).

\begin{rem}\em
By considering vector bundles as Lie groupoids with the multiplication given by fiberwise addition and $q=1$, the map $\mathcal{T}$ coincides with $\mathcal{V}$ given by \eqref{dfn:V_lift} and formulas \eqref{dfn:Dlr_mult} recover \eqref{dfn:D_l_r}. 
\end{rem}

The triple $(\mathfrak{D}, \mathfrak{l}, \mathfrak{r})$ is called the IM $(q,p)$-tensor corresponding to $\tau$ and it satisfy the Leibniz equation
\begin{equation}\label{eq:IM_leibniz}
 \mathfrak{D}(fa) = f \mathfrak{D}(a) + df \wedge l(a) - r(df) \wedge a, \,\,\, f \in C^\infty(M), \, a \in \Gamma(A)
\end{equation}
plus a set of compatibility equations called the \textit{IM equations}. It will not be necessary for us to recall the general form of the IM equations, except for the cases $p=q=1$ (endomorphisms) and $q=2, \,p=0$ (bivector fields). 

\paragraph{\bf Multiplicative endomorphisms.} A vector-valued 1-form  $K \in \Omega^1(\G, T\G)$ is multiplicative if and only if, seen as map $K: T\G \to T\G$, it is a Lie groupoid morphism. Let $(\mathfrak{D}, \mathfrak{l}, \mathfrak{r})$ be the corresponding IM (1,1) tensor on the Lie algebroid $A$ and consider the triple $\D = (D, l, r)$ given by $D= \mathfrak{D}, \,l=\mathfrak{l},\, r = \mathfrak{r}^*$. The Leibniz equation \eqref{eq:IM_leibniz} implies that  $\D \in \GDer^1(A)$ and, in this case, the IM equations are exactly the set of equations \eqref{IM1} - \eqref{IM4} of \S \ref{sec:der_on_algbds}. Note that the Lie algebroid morphism $K_A: TA \to TA$ obtained from differentiating $K$ is the linear vector-valued 1-form associated to $\D \in \GDer^1(A)$.

It will be also important to recall that the Nijenhuis torsion of a multiplicative endomorphism $K$, $N_K \in \Omega^2(\G, T\G)$, is also a multiplicative vector-valued 2-form and its corresponding IM (1,2) tensor is $\frac{1}{2}[\D, \D] \in \GDer^2(A)$.


%
%

\paragraph{\bf Multiplicative bivector fields.}
For a bivector field $\pi \in \Gamma(\wedge^2 T\G)$, notice that the corresponding IM $(2,0)$-tensor has only two components: $\mathfrak{D}: \Gamma(A) \to \Gamma(\wedge^2 A)$ and  $\mathfrak{r}:T^*M \to A$. Define $\delta = \mathfrak{D}$ and $\rho_*=\mathfrak{r}^*$. The Leibniz equation \eqref{eq:IM_leibniz} implies that there is a pre-Lie algebroid structure on $A^*$ defined by
\begin{equation}\label{eqn:delta_bracket}
\<[\mu_1, \mu_2]_*, a\> = \Lie_{\rho_*(\mu_1)}\<\mu_2, a\> - \Lie_{\rho_*(\mu_2)}\<\mu_2, a\> - \delta(a)(\mu_1, \mu_2),
\end{equation}
for $\mu_1, \, \mu_2 \in \Gamma(A^*), \, a \in \Gamma(A)$. The bracket $[\cdot,\cdot]_*$ will satisfy the Jacobi equation if and only if $\pi$ is a Poisson structure (i.e. $[\pi, \pi] =0$) (see \cite{ILX}). The IM-equations in this case reduces to 
\begin{equation}\label{eq:Lie_bialg}
\delta([a,b]) = [\delta(a), b] + [a, \delta(b)].
\end{equation}
So, the IM $(2,0)$-tensor is equivalent to the Lie bialgebroid structure on $(A,A^*)$ corresponding to the Poisson groupoid $(\G, \pi)$. 

\begin{rem}\label{rem:bivec_Lie}\em
It is important to note that, similar to the correspondence between IM (1,1)-tensors and Lie algebroid endomorphisms of the tangent prolongation $TA \to TM$, there is a linear bivector field $\pi_{A^*} \in \frakx^2(A)$ corresponding to the pre-Lie algebroid structure $[\cdot, \cdot]_*$ coming from the IM (2,0)-tensor. Also, the IM-equation \eqref{eq:Lie_bialg} is equivalent to $\pi_{A^*}^\sharp: T^*A \to TA$ being a Lie algebroid morphism from the cotangent Lie algebroid to the tangent Lie algebroid of $A$. We refer the reader to \cite{Mac-Xu} for further details (see Theorem 6.2. therein).
\end{rem}

\subsection{Definition and main properties}
Let $(A, A^*)$ be a Lie bialgebroid.

\begin{dfn}\em
A \textit{Lie-Nijenhuis bialgebroid structure} on $(A, A^*)$ is a generalized derivation $\D \in \GDer^1(A)$ of degree $1$ such that both $\D$ and $\D^\top$ are IM (1,1)-tensors and $[\D, \D]=0$. 
\end{dfn}

Due to Theorem \ref{thm:duality}, if $\D$ defines a Lie-Nijenhuis structure on $(A, A^*)$, then $\D^\top$ also defines a Lie-Nijenhuis structure on $(A^*, A)$. 

A Lie-Nijenhuis bialgebroid structure on $(A, A^*)$ is also equivalent to a linear vector-valued 1-form $K: TA \to TA$ such that both $K$ and its dual $K^\top:T(A^*) \to T(A^*)$ are Lie algebroid morphisms and the Nijenhuis torsion of $K$ being zero. Note that this automatically implies that the Nijenhuis torsion of $K^\top$ also vanishes because dualization is a graded Lie algebra isomorphism.

\begin{exam}\em
Let $(A, A^*)$ be a Lie bialgebroid over a connected manifold $M$ and consider $\D_\nabla \in \GDer^1(A)$ associated to a connection $\nabla: \Gamma(TM) \times \Gamma(A) \to \Gamma(A)$. It is straightforward to check that $D_\nabla$ defines an IM $(1,1)$-tensor if and only if $A$ is a bundle of Lie algebras (IM4 implies that $\rho \equiv 0$ and IM1 implies that parallel transport is a Lie algebra isomorphism). Similarly with $\D_\nabla^\top = \D_{\nabla^\top} \in \GDer^1(A^*)$. As $[\D_\nabla, \D_\nabla]=0$ is equivalent to  $\nabla$ being flat, one has that $\D_\nabla$ defines a Lie-Nijenhuis structure on $A$ if and only if $A\oplus A^* \to M$ is a flat Lie bialgebra bundle in the following sense: given a point $x_0 \in M$, the monodromy action of fundamental group $\pi_1(M, x_0)$ on $A_{x_0} \oplus A^*_{x_0}$ is by Lie bialgebra isomorphisms and
$$
A\oplus A^* \cong \widetilde{M} \times_{\pi_1(M, x_0)} (A_{x_0} \oplus A_{x_0}^*)
$$
as Lie bialgebroids, where $\widetilde{M}$ is the universal cover of $M$. We note that these structures of flat Lie bialgebra bundles have appeared in \cite{AN} in their study of the Goldman-Turaev Lie bialgebra.
\end{exam}

In \S \ref{sec:hol}, we will show that holomorphic Lie bialgebroids provide an important class of Lie-Nijenhuis bialgebroids.

For a Poisson manifold $(M, \pi)$, it is known that the cotangent Lie algebroid $(T^*M, [\cdot,\cdot]_\pi, \pi^\sharp)$ and $TM$ with its tangent Lie algebroid structure define a Lie bialgebroid \cite{Mac-Xu}. The next result show that Lie-Nijenhuis structures on $(TM, T^*M)$ correspond bijectively to Poisson-Nijenhuis structures on $M$.

\begin{prop}\label{prop:equiv_PN_LN}
A generalized derivation $\D \in \GDer^1(TM)$ with symbol $(l,r)$ defines a Lie-Nijenhuis bialgebroid structures on $(TM, T^*M)$ if and only if $(\pi, r)$ is a Poisson-Nijenhuis structure and $\D = \D^{r,T}$.
\end{prop}

\begin{proof}
From Corollary \ref{cor:tangent_1_1}, one obtains that $\D$ is an IM (1,1) tensor on $TM$ if and only if $\D = \D^{r,T}$. So, $\D^\top = D^{r, T^*}$ and Proposition \ref{prop:pn} implies that $\D^\top$ is an IM (1,1) tensor on $T^*M$ if and only if $\pi$ and $r$ are compatible. The vanishing of the Nijenhuis torsion of $r$ follows from  $[\D^{r,T}, \D^{r,T}] = D^{[r,r], T}$ (see Example \ref{exam:tang_D}).
\end{proof}

\begin{prop}\label{prop:LN_bialg}
 If $(A,A^*, \D)$ is a Lie-Nijenhuis bialgebroid and $\pi \in \frakx^2(M)$ is the associated Poisson structure defined by $\pi^\sharp = \rho_* \circ \rho^*$, then 
 \begin{itemize}
  \item[(i)] $(\pi, r)$ is a Poisson-Nijehuis structure on $M$;
  \item[(ii)] the bracket $[a,b]_l = [l(a), b] + [a,l(b)]-l([a,b])$ is a Lie algebroid bracket on $A$ for which $((A,[\cdot,\cdot]_l, \rho \circ l), (A^*, [\cdot,\cdot]_*, \rho_*)$ is a Lie bialgebroid.
  \item[(iii)] $((A,[\cdot,\cdot]_l, \rho \circ l), (A^*, [\cdot,\cdot]_*, \rho_*), \D)$ is a Lie-Nijenhuis bialgebroid.
 \end{itemize}
\end{prop}

\begin{proof}
The equality $r \circ \pi^\sharp = \pi^\sharp \circ r^*$ is a direct consequence of \eqref{IM4} for both $\D$ and $\D^\top$. Now, by dualizing \eqref{IM3_alt} for $(D, l, r)$, one obtains that $D^\top_X(\rho^*(\alpha)) = \rho^*(D^{r,T^*}_X(\alpha))$, for any $\alpha \in \Omega^1(M)$ and $X \in \frakx(M)$. Hence, by using \eqref{IM3_alt} for $\D^\top$, one has that \begin{align*}
\pi^\sharp(D^{r,T}_X(\alpha)) = \rho_*(D^\top_X(\rho^*(\alpha))) = D^{r,T}_X(\pi^\sharp(\alpha)),
\end{align*}
which, from \eqref{eq:mm_gen_der}, is exactly the vanishing of the Magri-Morosi concomitant. So, $\pi$ and $r$ are compatible. Finally, as the Nijenhuis torsion of $r$ is  the $TM$-component of $\frac{1}{2}[\D,\D]=0$, it follows that $(\pi, r)$ is a Poisson-Nijenhuis structure on $M$.  

As for (ii), notice that the Nijenhuis torsion $N_l$ of $l$ also vanishes. Indeed, by \eqref{IM2},
\begin{align*}
N_l(a,b)& := [l(a),l(b)]- l\left([l(a),b]+[a,l(b)]-l([a,b])\right)\\
& = l(D_{\rho(b)}(a)) - D_{\rho(b)}(l(a)) =0,
\end{align*}
where the last equality comes from the fact that $l(D_X(a))- D_X(l(a))$ is exactly the $A$-component of $\frac{1}{2}[\D,\D]=0$. So, $[\cdot,\cdot]_l$ is a Lie algebroid bracket \cite[Cor.~1.1]{KM}. By \eqref{IM2_alt} and Lemma \ref{lem:Lie_D}, one has that the linear Poisson structure on $A^*$ corresponding to $[\cdot,\cdot]_l$ is $\pi_K^\sharp = K^\top \circ \pi_{A}^\sharp$, where $\pi_A$ is the linear Poisson structure associated to $[\cdot,\cdot]$ on $A$ and $K^\top: T(A^*) \to T(A^*)$ is the linear endomorphism associated to $\D^\top$. As both maps are 
Lie algebroid morphisms, the result now follows from the characterization of Lie bialgebroids discussed in Remark \ref{rem:bivec_Lie}.

As for (iii), notice that one has only to check that $\D$ is an IM (1,1) tensor on $(A, [\cdot,\cdot]_l, \rho \circ l)$. We prove \eqref{IM1} and leave the other IM equations to the reader. First, using that $[\cdot, \cdot]_l = [\cdot,\cdot]_\D$ (because of \eqref{IM2_alt} for $\D$ on $(A, [\cdot,\cdot], \rho)$), one has that 
\begin{align*}
\mathcal{Q}&:=D_X([a,b]_l) - [D_X(a),b]_l - [a,D_X(b)]_l - D_{[\rho(l(b)),X]}(a) + D_{[\rho(l(a)),X]}(b) &  \\
&  = D_X([l(a),b]) - [D_X(l(a)),b] - [l(a),D_X(b)] + D_{[\rho(l(a)),X]}(b)\\
&  + D_X(D_\rho(b)(a)) - D_{\rho(b)}(D_X(a)) - D_{\rho(D_X(b))}(a) - D_{[r(\rho(b)), X]}(a)
\end{align*}
where we have used that $D_X(l(a)) = l(D_X(a))$ and $\rho \circ l = r \circ \rho$. By summing and subtracting $D_{[\rho(b), X]}(l(a))$ and using \eqref{IM1} and \eqref{IM3_alt} for $\D$ on $(A, [\cdot, \cdot], \rho)$, we obtain that 
\begin{align*}
\mathcal{Q} & =  D_X(D_{\rho(b)}(a)) - D_{\rho(b)}(D_X(a)) - D_{[r(\rho(b)), X]+[\rho(b), r(X)]- r([\rho(b),X])}(a)\\ 
& \hspace{-10pt}+ l(D_{[\rho(b), X]})(a) = D^2_{(\rho(b), X)}(a)
\end{align*}
where $D^2: \Gamma(A) \to \Gamma(\wedge^2 T^*M \otimes A)$ is the $D$-component of $\frac{1}{2}[\D, \D]$ (see e.g. \cite[Cor.~6.3]{BD}). The result now follows from $[\D, \D]=0$.
\end{proof}

\begin{rem}\em
 It is important to note that the definition of a Lie-Nijenhuis bialgebroid is symmetric with respect to $A$ and $A^*$. So, Proposition \ref{prop:LN_bialg} also implies that $((A, [\cdot, \cdot], \rho), (A^*, [\cdot,\cdot]_{l^*}, \rho_*\circ l^*), \D)$ is a Lie-Nijenhuis bialgebroid. Note that, in this way, we can generate a hierarchy of Lie-Nijenhuis bialgebroids
 $$
 ((A, [\cdot, \cdot], \rho), (A^*, [\cdot, \cdot]_{(l^*)^j}, \rho_*\circ (l^*)^j), \D)
 $$
\end{rem}

\begin{rem}\label{rem:courant}\em
For a Lie-bialgebroid $(A,A^*)$, there is a notion of Nijenhuis operator on its double $\mathbb{A}:=A\oplus A^*$. It is a vector bundle morphism $T: \mathbb{A} \to \mathbb{A}$ for which a Nijenhuis torsion constructed with the Courant bracket on $\mathbb{A}$ vanishes (see \cite{Kos2} and references therein). It follows from \cite[Thm.~4.5]{Kos2} and the proof of Proposition \ref{prop:LN_bialg}(ii) that if $\D = (D, r, l)$ is a Lie-Nijenhuis structure on $(A, A^*)$ such that $l^2= \lambda \mathrm{id}_A$, for some $\lambda \in \R$, then 
$$
T = 
\left(
\begin{matrix}
l & 0\\
0 & -l^*\\
\end{matrix}
\right)
$$
is a Nijenhuis operator on $\mathbb{A}$. 
\end{rem}

It is interesting to note that the formalism developed so far provides an alternative proof for a fundamental result of Y. Kosmann-Schwarzbach.

\begin{prop}\label{prop:alt_proof_kos}\cite{Kos}
Let $r \in \Omega^1(M, TM)$ and $\pi \in \frakx^2(M)$ be a Poisson manifold. Consider the cotangent Lie algebroid $T^*M_\pi=(T^*M,[\cdot,\cdot]_\pi, \pi^\sharp)$ and the pre-Lie algebroid $TM_r=(TM, [\cdot, \cdot]_r, r)$, where
$$
[X,Y]_r = [r(X), Y] + [X, r(Y)] - r([X,Y]),
$$
One has that $(\pi,r)$ is a Poisson-Nijenhuis structure if and only if $(TM_r, T^*M_\pi)$ is a Lie bialgebroid.
\end{prop}

\begin{proof}
First note that $TM_r$ is a Lie algebroid if and only if the Nijenhuis torsion of $r$ vanishes (see \cite[Cor.~1.1]{Kos}). Now, on the one hand, if $(\pi, r)$ is Poisson-Nijenhuis, then Propositions \ref{prop:equiv_PN_LN} and \ref{prop:LN_bialg} gives the Lie bialgebroid structure on $(TM_r, T^*M_\pi)$. On the other hand, let us assume that $(TM_r, T^*M_\pi)$ is a Lie bialgebroid. As recalled in Remark \ref{rem:bivec_Lie}, this is equivalent to the linear Poisson 
$$
\pi_{TM_r}^\sharp: T^*(T^*M) \to T(T^*M)
$$ 
corresponding to $TM_r$ being a Lie algebroid morphism (the Lie algebroid structures here are the cotangent and tangent Lie algebroids associated to $T^*M_\pi$). Now, by combining \eqref{eq:deform_bracket} with Lemma \ref{lem:Lie_D} and Theorem \ref{thm:lifts}, we get that 
$$
\pi_{TM_r}^\sharp = r^{\rm ctg} \circ \pi_{\rm can}^\sharp,
$$ 
where $\pi_{\rm can} \in \frakx^2(T^*M)$ is the Poisson structure corresponding to the canonical symplectic form on $T^*M$ (see also  \cite[Lem.~7.1]{BP}). Here, we have used that $\pi_{\rm can}$ is the linear Poisson structure associated to the canonical Lie algebroid structure on the tangent bundle $TM$. The argument now proceeds as follows: the fact that $(TM, T^*M_\pi)$ is a Lie algebroid for any Poisson structure $\pi$ implies that $\pi_{\rm can}^\sharp: T^*(T^*M) \to T(T^*M)$ is a Lie algebroid morphism, so $r^{\rm ctg} = \pi_{TM_r}^\sharp \circ (\pi_{\rm can}^\sharp)^{-1}$ is also a Lie algebroid morphism. The result now follows from Corollary \ref{cor:pn_ctg_char}.
\end{proof}

\subsection{Lie theory of Poisson-Nijenhuis groupoids}\label{sec:pn_grpd}
We now proceed to study Poisson-Nijenhuis structures on Lie groupoids and their infinitesimal data. Our data will be $(\G,\pi, K)$, where $\G \toto M$ is a source-connected and source-simply connected groupoid endowed with a multiplicative Poisson structure $\pi \in \frakx^2(\G)$ and a multiplicative vector-valued 1-form $K \in \Omega^1(\G, T\G)$. Let $\D \in \GDer^1(A)$ be the corresponding IM (1,1)-tensor on the Lie algebroid $A \to M$ corresponding to $K$ and consider the Lie algebroid structure $([\cdot,\cdot]_*, \rho_*)$ on $A^*$ making $(A, A^*)$ the Lie bialgebroid corresponding to $\pi$.

\begin{thm}\label{thm:pn_grpd}
The Poisson structure $\pi$ and $K$ are compatible if and only if $\D^\top \in \GDer^1(A^*)$ is an IM (1,1) tensor on the Lie algebroid $A^*$. In particular, there is a 1-1 correspondence between multiplicative Poisson-Nijenhuis structures on $\G$ and Lie-Nijenhuis bialgebroid structures on $(A, A^*)$. In this case, the Lie bialgebroid corresponding to $\pi_K^\sharp = K\circ \pi^\sharp$ is $(A^*, \rho_* \circ l^*, [\cdot, \cdot]_{l^*})$, where
$$
[\mu_1, \mu_2]_{l^*} = [l^*(\mu_1), \mu_2]_* + [\mu_1,  l^*(\mu_2)]_* - l^*([\mu_1,\mu_2]_*).
$$
\end{thm}

Our proof will be based on a detailed analysis of the Magri-Morosi concomitant on the Lie groupoid $\G$ which we present below in a serie of Lemmas and one Proposition. First of all, recall that $R_\pi^K$ will be a tensor if and only if $\pi^\sharp \circ K^* = K \circ \pi^\sharp$. The next result generalizing Lemma \ref{lem:Lie_D} give the infinitesimal conditions for this to happen.

\begin{lem}\label{lem:IM4}
 One has that $\pi^\sharp \circ K^* = K \circ \pi^\sharp$ if and only if the pseudo Lie-bracket (see \eqref{dfn:pseudo_lie})
 $$
 [\mu_1, \mu_2]_{\D^\top} = [l^*(\mu_1), \mu_2]_* + D^\top_{\rho(\mu_2)}(\mu_1) 
 $$
 is skew-symmetric. In particular, $\rho_* \circ l^* = r \circ \rho_*$ and $([\cdot,\cdot]_{\D^\top}, \rho_*\circ l^*)$ is the pre-Lie algebroid structure on $A^*$ corresponding to $K \circ \pi^\sharp$.
\end{lem}

\begin{proof}
 Define $\delta_\D, \, \delta_K: \Gamma(A^*) \to \Gamma(A^* \otimes A^*)$ as follows:
 \begin{align*}
 \delta_\D(a)(\mu_1,\mu_2) & = \Lie_{\rho_*(l^*(\mu_1)}\<\mu_2, a\> - \Lie_{r(\rho_*(\mu_2)}\<\mu_1, a\> - \<[\mu_1,\mu_2]_{\D^\top}, a\>\\
 \delta_K(a)(\mu_1, \mu_2) & = \delta(a)(\mu_1, l^*(\mu_2)) - \<\mu_1, D_{\rho_*(\mu_2)}(a)\>,
 \end{align*}
 where $\delta: \Gamma(A) \to \Gamma(\wedge^2 A)$ is the Chevalley-Eilenberg differential of $A^*$. It is clear that $[\cdot, \cdot]_{\D^\top}$ is skew-symmetric if and only if $\rho_* \circ l^* = r \circ \rho_*$ and $\delta_D(a) \in \Gamma(\wedge^2 A)$. A straightforward calculation shows that
 $$
 \delta_\D(a)(\mu_1,\mu_2) = - \delta_K(a)(\mu_2, \mu_1),
 $$
The result now follows from \cite[Lem~6.9]{BD}.
\end{proof}

\begin{rem}\label{rem:lieD_from_grpd}\em
Given a Lie algebroid $A$, we can consider $A^*$ as a Lie groupoid with the multiplication given as fiberwise addition. In this case, the linear Poisson structure $\pi_A \in \frakx^2(A^*)$ is multiplicative and Lemma \ref{lem:Lie_D} follows directly from Lemma \ref{lem:IM4}. 
\end{rem}

From now on, we assume that $\pi^\sharp \circ K^* = K \circ \pi^\sharp$ on $\G$ and let us consider the $(2,1)$-tensor $\tau_\pi^K \in \Gamma(T^*\G \otimes \wedge^2 T\G)$ on $\G$ defined by
$$
\tau_\pi^K(U, \xi_1, \xi_2) = \<\xi_2, R^K_\pi(U, \xi_1)\> = \<U, C^K_\pi(\xi_1, \xi_2)\>
$$
We will be interested in finding infinitesimal conditions for $\tau_\pi^K$ to vanish. Our goal is to show that $\tau_\pi^K$ is a multiplicative tensor and to relate the vanishing of the corresponding IM $(2,1)$-tensor to the IM equations $\D^\top$ has to satisfy. So, let us define:
\begin{align*}
 \mathcal{Q}_r(X, \mu)& := \rho_*(D_X^\top(\mu)) - D_X^{r,T}(\rho_*(\mu))\\
 \mathcal{Q}_l(\mu_1, \mu_2) & :=  [\mu_1,l^*(\mu_2)]_*- D^\top_{\rho_*(\mu_2)}(\mu_1) - l^*([\mu_1,\mu_2]_*)\\
 \mathcal{Q}_D(X,\mu_1, \mu_2) & := D_X^\top([\mu_1, \mu_2]_*) -[D_X^\top(\mu_1), \mu_2] - [\mu_1, D^\top_X(\mu_2)]\\
& \hspace{-30pt} - D^\top_{[\rho_*(\mu_2), X]}(\mu_1) + D^{\top}_{[\rho_*(\mu_1), X]}(\mu_2)
%
\end{align*}

We will need to introduce the notion of projectable 1-forms and vector fields on $\G$. We say that a 1-form $\xi \in \Gamma(T^*\G)$ (resp. vector field $U \in \Gamma(T\G)$) is \textit{projectable} if there exists $\mu \in \Gamma(A^*)$ (resp. $X \in \Gamma(TM)$) such that $\widetilde{t}\circ \xi = \mu \circ t$ (resp. $Tt \circ U = X \circ t$). Here, $\widetilde{t}: T^*\G \to A^*$ is the target map of the cotangent groupoid. Recall that, 
$$
 \<\widetilde{t}(\xi), a\> = \<\xi, \overrightarrow{a}\>.
$$

\begin{lem}\label{lem:tau_lr}
For $a \in \Gamma(A), \, \alpha \in \Gamma(T^*M)$, 
$$
 i_{t^*\beta} \tau_\pi^K = \mathcal{T}(Q_r(\beta)), \,\,\,\,
 \tau_\pi^K(\overrightarrow{a}, \cdot) = \mathcal{T}(Q_l(a)),
$$
where $Q_r: T^*M \to T^*M \otimes A$ and $Q_l: A \to \wedge^2 A$ are given by 
$$
Q_r(\beta)|_{(X,\mu)}  = \<\beta, \mathcal{Q}_r(X,\mu)\>, \,\,\, Q_l(a)|_{(\mu_1,\mu_2)} = \<\mathcal{Q}_l(\mu_1, \mu_2), a\>\\
$$
\end{lem}

\begin{proof}
 For projectable vector fields $U, \, V \in \Gamma(TM)$, note that $D^{K, T}_U(V)$ is also projectable over $D^{r, T}_X(Y)$, where $X, \,Y$ are the projections of $U, \,V$ respectively. Also,
 \begin{equation}\label{eq:proj_D_ctg}
 \widetilde{t} \circ D^{K,T^*}_U(\xi) = D^{\top}_X(\mu) \circ t,
 \end{equation}
 for projectable 1-form $\xi \in \Gamma(T^*M)$ with projection $\mu \in \Gamma(A^*)$. Indeed, it follows from \eqref{eqn:1_duality} and \eqref{dfn:Dlr_mult}.
 Using that $Tt \circ \pi^\sharp = - \rho_* \circ \widetilde{t}$, it is now straightforward to check that  $i_{t^*\beta} \tau_\pi^K = \mathcal{T}(Q_r(\beta))$.
 
 Let now $a \in \Gamma(A)$. By recombining the terms on $R_\pi^K$ and using that $\Lie_U (K^*(\xi)) = (\Lie_UK)^*(\xi) + K^*(\Lie_U\xi)$, one obtains
 \begin{align*}
  R_\pi^K(\overrightarrow{a}, \cdot) = (\Lie_{K(\overrightarrow{a})} \pi)^\sharp(\cdot) - K((\Lie_{\overrightarrow{a}}\pi)^\sharp(\cdot)) + \pi^\sharp((\Lie_{\overrightarrow{a}}K)^*(\cdot)).
  \end{align*}
  By choosing projectable $\xi_1, \, \xi_2 \in \Gamma(T^*\G)$, one obtains
 \begin{align*}
 \tau_\pi^K(\overrightarrow{a}, \xi_1, \xi_2)& = t^*(\delta(l(a))(\mu_1, \mu_2) - \delta(a)(\mu_1, l^*(\mu_2)) + \<\mu_1, D_{\rho_*(\mu_2)}(a)\>)\\
 & = \<[\mu_1, l^*(\mu_2)]_* - l^*([\mu_1,\mu_2]_*), a\> \\
 & \hspace{-30pt} -(\Lie_{\rho_*(\mu_2))}\<\mu_1, l(a)\> - \Lie_{r(\rho_*(\mu_2))}\<\mu_1, a\> - \<\mu_1, D_{\rho_*(\mu_2)}(a))\> \\ 
 & = t^* \<[\mu_1, l^*(\mu_2)]_* - l^*([\mu_1,\mu_2]_*) - D^\top_{\rho_*(\mu_2)}(\mu_1)), a\> \\
 & = t^*(Q_l(a)(\mu_1, \mu_2)).
 \end{align*}
 where $\delta: \Gamma(A) \to \Gamma(\wedge^2 A)$ is the Chevalley differential of $A^*$ and we have used \eqref{eqn:1_duality}, \eqref{dfn:Dlr_mult} for both $\pi$ and $K$, and \eqref{eqn:delta_bracket}. 
\end{proof}

Our main technical result which will imply Theorem \ref{thm:pn_grpd} is the following Lemma.
\begin{prop}\label{prop:Lie_tau}
$$
\Lie_{\overrightarrow{a}}\tau_\pi^K =\mathcal{T}(Q_D(a))
$$
where
\begin{align*}
Q_D(a)|_{(X, \mu_1, \mu_2)} & = \<a, \mathcal{Q}_D(\mu_1,\mu_2)\> + \Lie_X \<\mathcal{Q}_l(\mu_1, \mu_2), a\>\\
& \hspace{-40pt} + \Lie_{\mathcal{Q}_r(\mu_1,X)}\<\mu_2, a\> - \Lie_{\mathcal{Q}_r(\mu_2,X)}\<\mu_1, X\> 
\end{align*}
\end{prop}

The proof of Proposition \ref{prop:Lie_tau} will depend on two lemmas.

\begin{lem}\label{lem:MM1}
 Given $(\pi, r)$ on $M$, one has that
 $$
 \Lie_X C_\pi^r = C_{[X,\pi]}^r + C_{\pi}^{[X,r]}.
 $$
\end{lem}

\begin{proof}
First note that,
\begin{align*}
(\Lie_X \mathcal{R}^r_\pi)(Y,\alpha) & = [X, \mathcal{R}^r_\pi(Y,\alpha)] - \mathcal{R}^r_\pi([X,Y], \alpha) - \mathcal{R}^r_\pi(Y,\Lie_X\alpha)\\
& = A - B
\end{align*}
where
\begin{align*}
A & = [X, \pi^\sharp(D^{r, T^*}_Y(\alpha))] - \pi^\sharp(D_{[X,Y]}^{r,T^*}(\alpha)) - \pi^\sharp(D^{r, T^*}_Y(\Lie_X\alpha))\\
  & = [X, \pi]^\sharp(D^{r, T^*}_Y(\alpha)) + \pi^\sharp(\Lie_X D_Y^{r,T^*}(\alpha) - D_{[X,Y]}^{r,T^*}(\alpha) - D_Y^{r,T^*}(\Lie_X\alpha))\\ 
B & = [X, D^{r,T}_{Y}(\pi^\sharp(\alpha))] - D^{r,T}_{[X,Y]}(\pi^\sharp(\alpha)) - D^{r,T}_Y(\pi^\sharp(\Lie_X\alpha).
\end{align*}
Now, for any $Z \in \mathfrak{X}(M)$, using the Jacobi identity for the Fr\"olicher-Nijenhuis bracket,
\begin{align*}
 [X,D^{r,T}_Y(Z)]-D_{[X,Y]}^{r,T}(Z) & = [X, [Z,r](Y)] - [Z,r]([X,Y])\\
  & = [X,[Z,r]](Y)\\
 & = D^{r,T}_Y([X,Z]) + D^{[X,r],T}_Y(Z).
\end{align*}
By letting $Z=\pi^\sharp(\alpha)$, one has that
$$
B = D^{[X,r],T}_Y(\pi^\sharp(\alpha)) + D_{Y}^{r,T}([X,\pi]^\sharp(\alpha))
$$
Similarly, by using that $[X,r]^*(\alpha) = \Lie_X r^*(\alpha) -r^*(\Lie_X\alpha)$, one can write
\begin{align*}
D^{r,T^*}_Y(\Lie_X\alpha) & = \Lie_Y\Lie_X(r^*\alpha) - \Lie_Y [X,r]^*(\alpha) - \Lie_{r(Y)}\Lie_X\alpha
\end{align*}
Hence, $\Lie_X D_Y^{r,T^*}(\alpha) - D_{[X,Y]}^{r,T^*}(\alpha) - D_Y^{r,T^*}(\Lie_X\alpha)$ equals
\begin{align*}
 \Lie_{[r(Y),X]} \alpha +  \Lie_{r([X,Y])} \alpha + \Lie_Y [X,r]^*\alpha = D^{[X,r],T^*}_Y(\alpha)
\end{align*}
and
$$
A = [X, \pi]^\sharp(D^{r, T^*}_Y(\alpha)) + \pi^\sharp(D^{[X,r],T^*}_Y(\alpha)).
$$
This concludes the proof.
\end{proof}

\begin{lem}\label{lem:MM2}
 Let $\gamma \in \Gamma(\wedge^2 A)$ and consider $\overrightarrow{\gamma} \in \Gamma(\wedge^2 TG)$. For projectable $\xi \in \Gamma(T^*M)$ and $U \in \Gamma(T\G)$, one has that
 $$
 R_{\overrightarrow{\gamma}}^K (\xi, U) = \overrightarrow{\gamma^\sharp(D_X^\top(\mu)) - D_X(\gamma^\sharp(\mu))},
 $$
 where $X \in \frakx(M)$, $\mu \in \Gamma(A^*)$ are the projections of $U$ and $\xi$, respectively, and $\gamma^\sharp: A^* \to A$ is the contraction map corresponding to $\gamma$.
\end{lem}

\begin{proof}
 First note that $\overrightarrow{\gamma}^\sharp(\xi) = \overrightarrow{\gamma^\sharp(\mu)}$. From \eqref{eq:proj_D_ctg}, 
 $$
 \overrightarrow{\gamma}^\sharp(D^{K, T^*}_U(\xi)) = \overrightarrow{\gamma^\sharp(D^\top_X(\mu))}.
 $$
 Also, from \eqref{dfn:Dlr_mult}
 $$
 D^{K,T}_U(\overrightarrow{\gamma}^\sharp(\xi)) = D^{K, T}_U(\overrightarrow{\gamma^\sharp(\mu)}) = [\overrightarrow{\gamma^\sharp(\mu)}, K](U) = \overrightarrow{D_X(\gamma^\sharp(\mu))}.
 $$
 This concludes the proof.
\end{proof}

\begin{proof}[Proof of Prop.~\ref{prop:Lie_tau}]
From Lemma \ref{lem:MM1} and \eqref{dfn:Dlr_mult},
$$
 \Lie_{\overrightarrow{a}}R^{K}_\pi = R_{[\overrightarrow{a}, \pi]}^K + R_{\pi}^{[\overrightarrow{a},K]} = C_{\overrightarrow{\delta(a)}}^K + C_\pi^{\mathcal{T}(D(a))}
 $$
 Now, for projectable $U \in \mathfrak{X}(\G)$ and $\xi_1,\, \xi_2 \in \Gamma(T^*\G)$ with projections $X \in \frakx(M)$ and $\mu_1, \mu_2 \in \Gamma(A^*)$, respectively, one has that
 $$
 \<\xi_2, R_{\pi}^{\mathcal{T}(D(a))}(U, \xi_1)\> = t^*\Sigma,
 $$
 where
\begin{align*}
 \Sigma & = \Lie_X\<\mu_1, D_{\rho_*(\mu_2)}(a)\> + \<\mu_1,D_{[\rho_*(\mu_2),X]}(a)\> + \delta(D_X(a))(\mu_1,\mu_2)\\
 & \hspace{-15pt}- \<\mu_2, D_{[\rho_*(\mu_1), X]}(a)\>.
\end{align*}
Indeed, call $\mathcal{T}(D(a)) = \Phi: T\G \to T\G$. Note that $\Phi(U) = \overrightarrow{D_X(a)}$ and $Tt(\pi^\sharp(\xi_i))= - \rho_*(\mu_i)$, for $i=1,2$. Hence, from \eqref{dfn:Dlr_mult} for both $K$ and $\pi$, one has that
 \begin{align*}
 \<\xi_2, R_\pi^\Phi(U,\xi_1)\> &  = \<\xi_2, \pi^\sharp(\Lie_U \Phi^*(\xi_1)- \Lie_{\Phi(U)}\xi_1) - [\pi^\sharp(\xi_1), \Phi](U)\>\\
 & = \<\xi_2, \pi^\sharp(\Lie_U \Phi^*(\xi_1)) + [\Phi(U),\pi]^\sharp(\xi_1) + \Phi([\pi^\sharp(\xi_1), U])\>\\
 & = - \Lie_U\<\xi_1, \Phi(\pi^\sharp(\xi_2))\> + \<\xi_1, \Phi([U, \pi^\sharp(\xi_2)])\>  \\
 & \hspace{-20pt} t^*\left(\delta(D_X(a))(\mu_1,\mu_2) -\<\mu_2, D_{[\rho_*(\mu_1), X]}(a)\>\right)\\
 & = t^*\Sigma.
\end{align*}
Also, from Lemma \ref{lem:MM2}
$$
\<\xi_2, R_{\overrightarrow{\delta(a)}}^K(U, \xi_1)\> = t^*\left(\delta(a)(D_X^\top(\mu_1),\mu_2) - \<\mu_2, D_X(\delta(a)^\sharp(\mu_2))\>\right) =: t^*\Upsilon,
$$
Using \eqref{eq:dual_D} and \eqref{eqn:delta_bracket} repeatedly, we can re-arrange the terms on $\Upsilon$ as follows:
\begin{align*}
\Upsilon & = \mathcal{L}_{11}(\<\mu_1, a\>) + \mathcal{L}_{12}(\<\mu_1, l(a)\>) + \mathcal{L}_{21}(\<\mu_2,a\>) + \mathcal{L}_{22}(\<\mu_2,l(a)\>)\\
& \hspace{-10pt} + \Lie_X\<[\mu_1,l^*(\mu_2)],a\> - \Lie_{r(X)}\<[\mu_1,\mu_2]_*, a\> + \Lie_{\rho_*(\mu_2)}\<\mu_1, D_X(a)\>\\
& \hspace{-10pt} - \Lie_{\rho_*(\mu_1)}\<\mu_2, D_X(a)\> - \<[D_X^\top(\mu_1), \mu_2]_* + [\mu_1, D_X^\top(\mu_2)]_*, a\>
\end{align*}
where each $\mathcal{L}_{ij}$ is an operator given by the following formulas:
\begin{align*}
 \mathcal{L}_{11} & =  \Lie_X \Lie_{r(\rho_*(\mu_2))} + \Lie_{r([\rho_*(\mu_2), X])} - \Lie_{\mathcal{Q}_r(X,\mu_2)}\\
 \mathcal{L}_{12} & = - \Lie_X \Lie_{\rho_*(\mu_2)} - \Lie_{[\rho_*(\mu_2), X]}\\
 \mathcal{L}_{21} & = \Lie_{\mathcal{Q}_r(X, \mu_1)} - \Lie_{r([\rho_*(\mu_1), X])}\\
 \mathcal{L}_{22} & = \Lie_{[\rho_*(\mu_1),X]}
\end{align*}
Now, from 
\begin{align*}
 \Lie_{Y}\<\mu_i, l(a)\> - \Lie_{r(Y)}\<\mu_i, a\>  =  \<D_{Y}^\top(\mu_i), a\> + \<\mu_i, D_{Y}(a)\>,
\end{align*}
we can express the four first terms of $\Upsilon$ involving $\mathcal{L}_{ij}$ as:
\begin{align*}
& \Lie_{\mathcal{Q}_r(X, \mu_1)}\<\mu_2, a\> - \Lie_{\mathcal{Q}_r(X,\mu_2)}\<\mu_1, a\> - \Lie_X\<D^\top_{\rho_*(\mu_2)}(\mu_1), a\>  - \Lie_X\<\mu_1, D_{\rho_*(\mu_2)}(a)\>  \\ 
& \hspace{-5pt} - \<\mu_1, D_{[\rho_*(\mu_2), X]}(a)\>  - \<D^\top_{[\rho_*(\mu_2), X]}(\mu_1), a\>  + \<\mu_2, D_{[\rho_*(\mu_1), X]}(a)\> \\
& + \<D^\top_{[\rho_*(\mu_1), X]}(\mu_2), a\> 
\end{align*}
The term $\Lie_X\<D^\top_{\rho_*(\mu_2)}(\mu_1), a\>$ together with the fifth and sixth term of $\Upsilon$ gives
\begin{align*}
- \Lie_X\<D^\top_{\rho_*(\mu_2)}(\mu_1), a\>  + \Lie_X\<[\mu_1,l^*(\mu_2)],a\> - \Lie_{r(X)}\<[\mu_1,\mu_2]_*, a\> & =\\
& \hspace{-200pt} =  \Lie_{X}\<\mathcal{Q}_l(\mu_1,\mu_2), a\> + \Lie_X\<l^*([\mu_1, \mu_2]), a\> - \Lie_{r(X)}\<[\mu_1, \mu_2]_*,a\>\\
& \hspace{-200pt} = \Lie_X\<\mathcal{Q}_l(\mu_1,\mu_2), a\> + \<D_X^\top([\mu_1, \mu_2]_*), a\> + \<[\mu_1, \mu_2]_*, D_X(a)\>.
\end{align*}
Therefore, grouping the terms on $\mathcal{Q}_D$ together and using \eqref{eqn:delta_bracket} once again, $\Upsilon$ can be rewritten as follows:
\begin{align*}
\Upsilon = & \<\mu_2, D_{[\rho_*(\mu_1), X]}(a)\> - \delta(D_X(a))(\mu_1,\mu_2) - \Lie_X\<\mu_1,D_{\rho_*(\mu_2)}(a)\>\\
& \hspace{-15pt} - \<\mu_1, D_{[\rho_*(\mu_2),X]}(\mu_2)\> + \Lie_{\mathcal{Q}_r(X, \mu_1)}\<\mu_2, a\> - \Lie_{\mathcal{Q}_r(X,\mu_2)}\<\mu_1, a\>\\
& \hspace{-15pt} + \Lie_X\<\mathcal{Q}_l(\mu_1,\mu_2), a\> + \<\mathcal{Q}_D(X,\mu_1,\mu_2),a\>\\
 = & -\Sigma + Q_D
\end{align*}
Finally, as
$$
(\Lie_{\overrightarrow{a}}\tau_\pi^K)(U, \xi_1, \xi_2) = \<\xi_2, (\Lie_{\overrightarrow{a}} R_\pi^K)(U, \xi_1)\> = t^*(\Sigma+ \Upsilon) = t^*Q_D,
$$
the conclusion holds.
\end{proof}

The proof of Theorem \ref{thm:pn_grpd} follows directly from what we have obtained so far.
\begin{proof}[Proof of Theorem \ref{thm:pn_grpd}]
Since $\G \toto M$ is the source 1-connected groupoid integrating $A$, one has that Lie bialgebroid structures on $(A, A^*)$ correspond to multiplicative Poisson structures $\pi \in \frakx^2(\G)$ and IM (1,1)-tensors $\D \in \GDer^1(A)$ correspond to multiplicative $K: T\G \to T\G$. Moreover, the vanishing of the Nijenhuis torsion $N_K$ is equivalent to $[\D,\D]=0$ (see \cite[Cor.~6.3]{BD}). So, it remains to understand the compatibility of $\pi$ and $K$ infinitesimally. Since $\G$ is source-connected, it follows Lemma \ref{lem:tau_lr} and Proposition \ref{prop:Lie_tau} that whenever $\tau_\pi^K$ is a tensor (i.e. $\pi^\sharp \circ K^* = K \circ \pi^\sharp)$, it is a multiplicative tensor and its vanishing is equivalent to \eqref{IM1}, \eqref{IM2} and \eqref{IM3} for $\D^\top$. The result now follows from noting that \eqref{IM2_alt} together with Lemma \ref{lem:IM4} implies that $\tau_\pi^K$ is a tensor and that $[\cdot,\cdot]_{\D^\top} = [\cdot, \cdot]_{l^*}$. This concludes the proof.
\end{proof}

By considering symplectic groupoids, we are able to recover a result of \cite{StX} concerning the integration of Poisson-Nijenhuis structures.
 
\begin{cor}
 Let $(M, \pi)$ be an integrable Poisson manifold and let $(\G, \omega) \toto M$ be the source 1-connected symplectic groupoid integrating  $\pi$. There is a 1-1 correspondence between Poisson-Nijenhuis structures $(\pi, r)$ on $M$ and multiplicative symplectic-Nijenhuis structures $(\omega, K)$ on $\G$, where $K: T\G \to T\G$ is the multiplicative endomorphism integrating 
 $r^{\rm ctg}: T(T^*M) \to T(T^*M)$.
\end{cor}

\begin{proof}
 By considering the Poisson structure $\pi_\G = \omega^{-1}$, it is well-known that $(\G, \pi_\G)$ is the Lie groupoid integrating the Lie bialgebroid $(TM, T^*M)$ coming from $\pi \in \frakx^2(M)$ (see \cite[Thm.~5.3]{Mac-Xu2}). Since Poisson-Nijenhuis structures $(\pi, r)$ are in 1-1 correspondence to Lie-Nijenhuis bialgebroid structures $D^{r,T} \in \GDer^1(TM)$ on $(TM, T^*M)$, the result now is an immediate application of Theorems \ref{thm:lifts} and \ref{thm:pn_grpd}. 
\end{proof} 
 
\begin{cor}\label{cor:das_result}
Let $(A, [\cdot, \cdot], \rho)$ be a Lie algebroid and $K \in \Omega^1(A, TA)$ be a linear vector-valued 1-form. The map $K: TA \to TA$ is a Lie algebroid morphism if and only if its dual $K^\top: T(A^*) \to T(A^*)$ is compatible with the linear Poisson structure $\pi_A \in \frakx^2(A^*)$.
\end{cor}

\begin{proof}
 This follows immediately from Theorem \ref{thm:pn_grpd} by considering $(A^*, \pi_A)$ as a Poisson groupoid (see also Remark \ref{rem:lieD_from_grpd}).
\end{proof}

\begin{rem}\em
In \cite{Das}, the author defines the notion of PN Lie bialgebroids as a Lie bialgebroid $(A, A^*)$ endowed with linear vector-valued 1-form $K: TA \to TA$ such that (1) $K$ is a morphism of Lie algebroids; (2) $K^*: T^*A \to T^*A$ is a also a morphism of Lie algebroids and (3) $(K, \pi_{A^*})$ is a Poisson-Nijenhuis pair, where $\pi_{A^*} \in \frakx^2(A)$ is the linear Poisson bivector corresponding to the Lie algebroid structure on $A^*$. We notice that (2) is a redundant condition since $K$ being a morphism of Lie algebroids implies that $K^*$ is also a morphism of algebroids. Hence, from Corollary \ref{cor:das_result}, it follows that PN Lie bialgebroids coincide with Lie-Nijenhuis bialgebroids and Theorem \ref{thm:pn_grpd} recovers the main result in \cite{Das}.
\end{rem}

\section{Holomorphic Lie-bialgebroids}\label{sec:hol}
In this section, we show how the formalism of generalized derivations is well suited to study holomorphic geometric structures (e.g. vector bundles, Poisson structures, Lie groupoids) in the realm of real differential geometry. In particular, we revisit the infinitesimal-global correspondence between holomorphic Poisson groupoids and holomorphic Lie bialgebroids \cite{LSX} from this perspective.

\subsection{Holomorphic vector bundles}
Let $\E \to M$ be a holomorphic vector bundle and consider the Dolbeault operator $\overline{\partial}:  \Gamma(\E) \to \Omega^{0,1}(M,\E)$. Denote by $q:E \to M$ the underlying real vector bundle and by $l: E \to E$ the endomorphism corresponding to the fiberwise multiplication by $i$. Also, define $D: \Gamma(E) \to \Omega^1(M, E)$ as 
$$
D_X(u) := \Psi^{-1}(i\, \overline{\partial}_{X+ir(X)}(\Psi(u))) = l( \Psi^{-1} \circ \overline{\partial}_{X+ir(X)}\circ \Psi(u))
$$
where $\Psi: E \to \E$ is the natural $\R$-linear isomorphism, $X \in \frakx(M)$ and $r:TM \to TM$ is the complex structure on $M$. It is straightforward to check that $\D^\E = (D,l,r) \in \GDer^1(E)$. We shall refer to $\D^\E$ as the \textit{Dolbeault generalized derivation} associated to $\E$. A section $u \in \Gamma(E)$ is said to be \textit{holomorphic} if $\Psi(u)$ is holomorphic. Note that $u$ is holomorphic if and only if $D(u) = 0$. 

As a real manifold, $E$ inherits a complex structure $J: TE \to TE$, which is linear since the multiplication by real scalars $h_\lambda$ is a holomorphic map. 

\begin{prop}
The Dolbeault generalized derivation is the generalized derivation $\D^J$ corresponding to $J$ via \eqref{dfn:D_l_r}.
\end{prop}
\begin{proof}
Let $\{\sigma_1, \dots, \sigma_n\}$ be a local frame of holomorphic sections of $\E$ over an open set $U \subset M$ with holomorphic coordinates $(z_1, \dots, z_m): U \to \C^m$.  Define $u_j = \Psi^{-1}(\sigma_j)$, $v_j = \Psi^{-1}(i\,\sigma_j)$. It is straightforward to check that
\begin{align*}
J(\partial/\partial x_k) = \partial/\partial y_k ,& \,\,\,\,J(\partial/\partial y_k) = -\partial/\partial y_k\\
J(\partial/\partial \xi_j) = \partial/\partial \eta_j ,& \,\,\,\,J(\partial/\partial \eta_j) = -\partial/\partial \xi_j,
\end{align*}
where $z_k = x_k + i\,y_k$ and $(x_k,y_k,\xi_j, \eta_j)$ is the coordinate system on $q^{-1}(U) \subset E$ corresponding to the frame $\{u_j, v_j\}$. As $r(\partial/\partial x_k) = \partial/\partial y_k$, $u_j^\uparrow = \partial/\partial \xi_j,\, v_j^\uparrow = \partial/\partial \eta_j$ and $l(u_j)=v_j$, it follows that $(l,r)$ is the symbol of $J$.
Also,
\begin{align*}
D^J_{\frac{\partial}{\partial x_k}}(u_j)^\uparrow = [u_j^\uparrow, J(\frac{\partial}{\partial x_k})]- J([u_j^\uparrow, \frac{\partial}{\partial x_k}]) = [\frac{\partial}{\partial \xi_j}, \frac{\partial}{\partial y_k}] - J([\frac{\partial}{\partial \xi_j},  \frac{\partial}{\partial x_k}]) = 0.
\end{align*}
Similarly, one proves that $D^J_{\partial/\partial x_k}(v_j)=0$ . Hence, for $f, g \in C^\infty(M)$, by the Leibniz rule,
\begin{align*}
D_{\frac{\partial}{\partial x_k}}^J(f u_j +  g v_j) & = (\Lie_{\frac{\partial}{\partial x_k}}f) v_j -  (\Lie_{\frac{\partial}{\partial y_k}}f) u_j  -  (\Lie_{\frac{\partial}{\partial x_k}}f) u_j +  (\Lie_{\frac{\partial}{\partial y_k}}f) v_j\\
& = \Psi^{-1}\left(i\,\Lie_{\frac{\partial}{\partial x_k}+ i\frac{\partial}{\partial y_k}} (f+ig)\, \sigma_j\right)\\
& = \Psi^{-1}\left(
i \,\bar{\partial}^E_{\frac{\partial}{\partial x_k}+ i\frac{\partial}{\partial y_k}}(\Psi(fu_j+gv_j))\right)
\end{align*}
as we wanted to prove. 
\end{proof}

The duality for the Dolbeault generalized derivations also coincides with the duality of holomorphic vector bundles.

\begin{prop}\label{prop:dual_complex}
Let $\E^* \to M$ be the holomorphic dual of $\E$. If $\D^\E \in \GDer^1(E)$ is the Dolbeault generalized derivation of $\E$, then its dual is the Dolbeault generalized derivation of $\E^*$, i.e. $(\D^\E)^\top = \D^{\E^*}$. In particular, the complex structure on $E^*$ is $J^\top$.
\end{prop}
\begin{proof}
For $x \in M$, $\E^*_x = \mathrm{Hom}_{\C}(\E_x, \C)$. The Dolbeault operator $\bar{\partial}^{E^*}: \Gamma(\E^*) \to \Omega^{0,1}(M, \E^*)$ on $\E^*$ is given by 
 $$
 \bar{\partial}^{E^*}_{X+ir(X)}(\varphi) (\sigma) = \Lie_{X+ir(X)}(\varphi(\sigma)) - \varphi(\bar{\partial}^E_{X+ir(X)}(\sigma)) 
 $$
 where $\sigma \in \Gamma(M, \E),\, \varphi \in \Gamma(M, \E^*),\, X \in \frakx(M)$ and $r: TM \to TM$ is the complex structure on $M$. The real vector bundle underlying $\E^*$ is naturally identified with $E^*$ as follows: define $\Phi: E^* \to \E^*$ as
 $$
 \Phi(\mu)(\sigma) = \<\mu, \Psi^{-1}(\sigma)\> - i \,\<\mu, l(\Psi^{-1}(\sigma))\>,
 $$
 where $\Psi: E \to \E$ is the $\R$-linear isomorphism corresponding to the underlying real vector bundle of $\E$. Its inverse is $\Phi^{-1}(\varphi) = \mathrm{Re}(\varphi \circ \Psi)$. So, the Dolbeault generalized derivation on $E^*$ is  given by
\begin{align*}
\<\Phi^{-1}(i\,\bar{\partial}^{E^*}_{X+ir(X)}(\Phi(\mu))), u \>&  = \mathrm{Re}\left( i\,\bar{\partial}^{E^*}_{X+ir(X)}(\Phi(\mu))(\Psi(u))\right)\\
& = -\mathrm{Im}\left(\Lie_{X+ir(X)}(\<\mu, u\>-i\<\mu, l(u)\>)\right)\\ 
& \hspace{-70pt}+ \mathrm{Im}\left(\<\mu, \Psi^{-1}\circ\bar{\partial}^{E}_{X+ir(X)}\circ\Psi(u)\> - i\<\mu, l(\Psi^{-1} \circ \bar{\partial}^{E}_{X+ir(X)}\circ\Psi(u))\>\right)\\
 & = \Lie_X\<\mu, l(u)\> - \Lie_{r(X)}\<\mu,u\> - \<\mu, D_X(u)\>\\
 & = \<D^\top_X(\mu),u\>.
\end{align*}
\end{proof}

Given a complex manifold $M$, consider its holomorphic tangent bundle $T^{1,0}M$. Our next result describes the corresponding linear complex structure on $TM$

\begin{prop}
 Let $r: TM \to TM$ be the complex structure of $M$. The linear complex structure on $TM$ (resp. $T^*M$) corresponding to the holomorphic vector bundle $T^{1,0}M$ is the tangent lift $r^{\rm tg}: T(TM) \to T(TM)$ (resp. cotangent lift $r^{\rm ctg}: T(T^*M) \to T(T^*M)$).
\end{prop}

\begin{proof}
 The Dolbeault operator $\bar{\partial}: \Gamma(T^{1,0}) \to \Gamma((T^{0,1})^* \otimes T^{1,0})$ is simply
$$
\bar{\partial}_{Y+ir(Y)}(X-ir(X)) = \mathrm{pr}^{1,0}[Y+i\,r(Y), X-i\,r(X)] = Z - ir(Z),
$$
where
\begin{align*}
Z
= r([r(Y),X] - r([Y,X])) = -r(D^{r,T}_Y(X))
\end{align*}
Hence, using the isomorphism $\Psi: TM \ni X \mapsto X - i\,r(X) \in T^{1,0}$, one sees that
$$
D^{r,T}_Y(X) = i \Psi^{-1}(\bar{\partial}_{Y+i\,r(Y)}(\Psi(X))).
$$
This shows that $r^{\mathrm{tg}}: T(TM) \to T(TM)$ is the complex structure on the total space of $TM$ corresponding to the holomorphic structure on $TM$. The result regarding the cotangent bundle follows directly from Proposition \ref{prop:dual_complex} and Theorem \ref{thm:lifts}.
\end{proof}

\subsection{Integration}
\medskip \subsubsection{Presentation as Lie-Nijenhuis bialgebroids}
A holomorphic Lie algebroid is a holomorphic vector bundle $\mathcal{A} \to M$ endowed with a holomorphic bundle map $\mathcal{P}: \mathcal{A} \to T^{1,0}M$ and a complex Lie algebra structure on the sheaf of holomorphic sections $\Gamma_{\rm hol}(\cdot, \mathcal{A})$ such that $\mathcal{P}$ induces a morphism of sheaves of complex Lie algebras from $\Gamma_{\rm hol}(\cdot, \mathcal{A})$ to $\Gamma_{\rm hol}(\cdot, T^{1,0}M)$ and  
$$
[\sigma_1, f \sigma_2] = f [\sigma_1, \sigma_2] + (\Lie_{\mathcal{P}(\sigma_1)}f) \sigma_2, \,\, \sigma_1, \, \sigma_2 \in \Gamma_{\rm hol}(U, \mathcal{A}), \, f \in \mathcal{O}(U),
$$
where $\mathcal{O}$ is the sheaf of holomorphic functions on $M$ and $U \subset M$ is a open set and $\Gamma_{\rm hol}$ refers to the sheaf of holomorphic sections.

For a holomorphic Lie algebroid $\mathcal{A} \to M$, consider the exterior algebra bundles $\wedge^\bullet_\C \mathcal{A}$, $\wedge^\bullet_\C \mathcal{A}^*$ (in the following, we shall drop the $\C$ subscript).  It is clear that, similar to the real case, one can define both a Chevalley-Eilenberg differential $d_\mathcal{A}: \Gamma_{\rm hol}(U, \wedge^\bullet \mathcal{A}^*) \to \Gamma_{\rm hol}(U, \wedge^{\bullet+1}\mathcal{A}^*)$ and a Schouten bracket 
$$
[\cdot,\cdot]_{\mathcal{A}}:\Gamma_{\rm hol}(U, \wedge^i \mathcal{A}) \times \Gamma_{\rm hol}(U, \wedge^j \mathcal{A}) \to \Gamma_{\rm hol}(U, \wedge^{i+j-1} \mathcal{A}).
$$

\begin{dfn}\em
A holomorphic Lie bialgebroid is a pair of holomorphic Lie algebroids in duality $\mathcal{A} \to M$, $\mathcal{A}^* \to M$ such that given holomorphic sections $\sigma_1, \sigma_2 \in \Gamma_{\rm hol}(U, \mathcal{A})$, one has that
\begin{equation}\label{eqn:hol_bialg}
d_{\mathcal{A}^*}[\sigma_1, \sigma_2]_{\mathcal{A}} = [d_{\mathcal{A}^*}\sigma_1, \sigma_2]_{\mathcal{A}} + [d_{\mathcal{A}^*}\sigma_1, \sigma_2]_{\mathcal{A}}.
\end{equation}
\end{dfn}

A holomorphic Lie algebroid structure on $\mathcal{A}$ determines a uniquely Lie algebroid structure $([\cdot,\cdot], \rho)$ on the underlying real vector bundle $A$ such that $\rho \circ l = r \circ \rho$ and $[\cdot, \cdot]$ restricts to a $\C$-linear bracket on the holomorphic sections of $A$ (see \cite[Prop.~3.3]{LMX2}). We shall refer to $A$ as the underlying real Lie algebroid of $\mathcal{A}$. 

\begin{rem}\label{rem:hol_alg_J}\em
Given an holomorphic vector bundle $\mathcal{A} \to M$ such that its underlying real vector bundle $A \to M$ has a Lie algebroid structure $([\cdot, \cdot], \rho)$, it is well known that there exists a holomorphic Lie algebroid structure on $\mathcal{A} \to M$ having $A \to M$ as the underlying real Lie algebroid if and only if the linear complex structure $J: TA \to TA$ is a Lie algebroid morphism (see \cite[Prop.~2.3]{LMX}).
\end{rem}

\begin{rem}\label{rem:hol_bialg_real}\em
Given two holomorphic Lie algebroids $\mathcal{A}, \, \mathcal{A}^*$ dual to each other, it can be proved that $(\mathcal{A}, \, \mathcal{A}^*)$ is a holomorphic Lie bialgebroid if and only if the underlying real Lie algebroids $(A, A^*)$ form also a real Lie bialgebroid  (see \cite[Thm.~4.10]{LSX}).
\end{rem}

We can summarize the remarks above about the properties inherited by the underlying real Lie bialgebroid of a holomorphic Lie bialgebroid using the language of Lie-Nijenhuis bialgebroids as follows: 

\begin{prop}\label{prop:hol_bialgbds}
 A Lie bialgebroid $(A, A^*)$ is the underlying real Lie bialgebroid of a holomorphic Lie bialgebroid $(\mathcal{A}, \mathcal{A}^*)$ if and only if there exists $\D=(D,l,r) \in \GDer^1(A)$ for which $(A, A^*, \D)$ is a Lie-Nijenhuis bialgebroid and
 \begin{equation}\label{eqn:alg_complex}
 r^2 = -\mathrm{id}_{TM}, \, l^2= - \mathrm{id}_E, \, D_{r(X)}(u) + l(D_X(u)) = 0,
 \end{equation}
 for $u \in \Gamma(E), \, X \in \mathfrak{X}(M)$. In this case, $\D = \D^\mathcal{A}$ is the Dolbeault generalized derivation corresponding to $\mathcal{A}$.
\end{prop}
%

\begin{proof}
It is known that the existence of a generalized derivation $\D \in \GDer^1(A)$ satisfying equations \eqref{eqn:alg_complex} is equivalent via the correspondence \eqref{dfn:D_l_r} to a linear vector valued form $J: TA \to TA$ satisfying $J^2 = - \mathrm{id}_{TA}$ (see \cite[Cor.~6.2]{BD}). The result now follows from Propositions \ref{prop:graded_lie_corresp}, \ref{prop:dual_complex} and Remarks \ref{rem:hol_alg_J} and \ref{rem:hol_bialg_real}.
\end{proof}

\medskip \subsubsection{Integration to Holomorphic Poisson groupoids}
Let us briefly recall the relationship between holomorphic Poisson structures and PN geometry. Let $(M, r)$ be a complex manifold and consider $\pi = \pi^0 - i \pi^1 \in \Gamma(\wedge^2 TM \otimes \C)$. It is known that
\begin{itemize}
 \item[(i)] $\pi \in \Gamma(\wedge^{2,0} TM)$ if and only if $\pi^1 =  \pi^0_r$;
 \item[(ii)] $\pi$ is a holomorphic (i.e. $\overline{\partial}\pi = 0$) if and only if $\pi^0$ and $r$ are compatible;
 \item[(iii)] $[\pi, \pi]=0$ if and only if $[\pi^0, \pi^0] =0$
\end{itemize}
(we refer to \cite{LMX2} for details). In the case (i), (ii) and (iii) are satisfied, we say that $(M, r, \pi)$ is a holomorphic Poisson manifold. We are now able to give an alternative proof of the result originally proved in \cite{LSX} (see Theorem 4.17 therein).

\begin{thm}\cite{LSX}\label{thm:holomorphic}
 Given a holomorphic Lie bialgebroid $(\mathcal{A}, \mathcal{A}^*)$, let $(A, A^*)$ be the underlying real Lie bialgebroid and consider $J_A: TA \to TA$ the linear complex structure on $A$. If $A$ is integrable and $\G \toto M$ is its source 1-connected groupoid, then $(\G, J, \pi)$ is a holomorphic Poisson manifold, where $J$ is the multiplicative complex structure integrating $J_A$ and $\pi = \pi^0 - i \pi^0_J$ is a holomorphic Poisson structure, where $\pi^0$ is the multiplicative Poisson structure integrating the Lie bialgebroid $(A, A^*)$.
\end{thm}

\begin{proof}
 Since $J_A: TA \to TA$ is a Lie algebroid morphism, it can be integrated to a multiplicative morphism $J: T\G \to T\G$. Also, as both $J$ and $J_A$ correspond to the same IM (1,1) tensor on $A$ (the Dolbeault generalized derivation $\D^\mathcal{A}$), it follows that $J^2 = -\mathrm{id}_{T\G}$ and $[J,J]=0$. Now, as $J_A^\top: T(A^*) \to T(A^*)$ is the complex structure corresponding to the holomorphic vector bundle $\mathcal{A}^*$, it is also a Lie algebroid morphism. So, Theorem \ref{thm:pn_grpd} implies that $\pi^0$ and $J$ are compatible giving the holomorphic Poisson structure $\pi = \pi^0 - i \pi^0_J$ on $(\G, J)$.
\end{proof}


\begin{thebibliography}{99}
\bibitem{AN}
Alekseev, A.; Naef, F.: Goldman-Turaev formality from the Knizhnik-Zamolodchikov connection, {\em Preprint  	arXiv:1708.03119.}

\bibitem{Bo}
Bonechi, F.: Multiplicative integrable models from Poisson-Nijenhuis structures, {\em Banach Cent. Publ.} {\bf 106} (2015), 19-33.


\bibitem{BP}
Broka, D.;  Xu, P.: Symplectic realizations of holomorphic Poisson manifolds, {\em Preprint 
arXiv:1512.08847.} (2015).

\bibitem{BD2}
Bursztyn, H.; Drummond, T.: Lie groupoids and the Fr\"olicher-Nijenhuis bracket, {\em Bull. Braz. Math. Soc.} {\bf 44} (2013), 709-730.


\bibitem{BD}
Bursztyn, H.; Drummond, T.: Lie theory of multiplicative tensors, {\em Math. Ann.} {\bf 375} (2019), 1489-1554.

\bibitem{BDN}
Bursztyn, H.; Drummond, T.; Netto, C.: In preparation.


\bibitem{BCO}
Bursztyn, H., Cabrera, A., Ortiz, C.: Linear and multiplicative 2-forms, {\em Lett. Math. Phys.} {\bf 90} (2009), 59-83.

\bibitem{Das}
Das, A.: Poisson-Nijenhuis groupoids, {\em preprint arXiv:1709.08168} (2017).

\bibitem{Dru}
Drummond, T.: In preparation.

\bibitem{FN}
Fr\"olicher, A., Nijenhuis, A.: Theory of vector valued differential
forms. Part I. {\em Indag. Math.} {\bf 18} (1956), 338-359.

\bibitem{GU}
Grabowski, J., Urbanski, P.: Tangent lift of Poisson and related structures, {\em J. Phys. A} {\bf 28} (1995), 6743-6777. 

\bibitem{GU2}
Grabowski, J., Urbanski, P.: Lie algebroids and Poisson-Nijenhuis structures, {\em Rep. Math. Phys.} {\bf 40} (1997), 195-208. 

\bibitem{GU3}
 Grabowski, J., Urbanski, P.: Tangent and cotangent lifts and graded Lie algebras associated with Lie algebroids, {\em Ann. Global Anal. Geom.} {\bf 15} (1997), 447-486.

\bibitem{ILX}
Iglesias Ponte, D., Laurent-Gengoux, C., Xu, P.: Universal lifting
theorem and quasi-Poisson groupoids, {\em J. Eur. Math. Soc.} {\bf 14} (2012), 681--731.

\bibitem{Kar}
Karasev, M.:  Analogues  of  objects  of  the  theory  of  Lie  groups  for  nonlinear
Poisson brackets, {\em USSR Izv.} {\bf 28} (1987), 497--527.


\bibitem{nat}
Kol\'ar, I., Michor, P., Slov\'ak, J.: Natural operations in differential geometry, {\em Springer-Verlag, Berlin Heidelberg 1993.}

\bibitem{Kos}
Kosmann-Schwarzbach, Y.: The Lie bialgebroid of a Poisson-Nijenhuis manifold, {\em Lett. Math. Phys.} {\bf 38} (1996), 421-428.

\bibitem{KM}
Kosmann-Schwarzbach, Y., Magri, F.,: Poisson-Nijenhuis structures. {\em Ann. Inst. H. Poincar\'e Phys. Th\'eor.} {\bf 53} (1990), no. 1, 35--81.

\bibitem{Kos2}
Kosmann-Schwarzbach, Y.: Nijenhuis structures on Courant algebroids. {\em Bull. Brazilian Math. Soc.} {\bf 42} (2011), 625-649.


\bibitem{LMX2}
Laurent-Gengoux, C., Sti\'enon, M., Xu, P.: Holomorphic Poisson manifolds and Lie algebroids, {\em IMRN: Int. Math. Res. Not.} (2008).

\bibitem{LMX}
Laurent-Gengoux, C., Sti\'enon, M., Xu, P.: Integration of holomorphic Lie algebroids. {\em Math. Ann.} {\bf 345}  (2009), 895--923.

\bibitem{LSX}
Lean, M.J., Sti\'enon, M., Xu, P.: Glanon groupoids. {\em Math. Ann.} {\bf 364} (2016), 485–518. 

\bibitem{Mac-book}
Mackenzie, K., General theory of Lie groupoids and Lie
algebroids. {\em London Mathematical Society Lecture Note Series} {\bf 213},
Cambridge University Press, Cambridge, 2005.

\bibitem{Mac-Xu}
Mackenzie, K., Xu, P.: Lie bialgebroids and Poisson groupoids.  {\em
Duke Math. J.}  {\bf 73}  (1994), 415--452.

\bibitem{Mac-Xu2}
Mackenzie, K., Xu, P.: Integration of Lie bialgebroids.  {\em
Topology}  {\bf 39}  (2000), 445--467.

\bibitem{MaMo}
Magri, F., Morosi, C.: Geometrical characterization of integrable Hamiltonian systems through the theory of Poisson-Nijenhuis structures, {\em Quaderno} \textbf{S19} (1984), University of Milan.

\bibitem{Mic}
Michor, P.: A generalization of Hamiltonian mechanics. {\em J. Geom. Phys.} {\bf 2} (1985), 67-82.

\bibitem{Pet}
Petalidou, F.: On the symplectic realization of Poisson-Nijenhuis manifolds. {\em Preprint arXiv:1501.07830v1} (2015).

\bibitem{Rawnsley}
Rawnsley, J.: Flat partial connections and holomorphic structures in $C^\infty$ vector bundles,
{\em Proc. Amer. Math. Soc.} {\bf 73} (1979), 391--397.

\bibitem{StX}
Sti\'enon, M., Xu, P.: Poisson quasi-Nijenhuis manifolds. {\em Comm. Math. Phys.} {\bf 270} (2007), 709--725.

\bibitem {We87}
{Weinstein, A.}: {Symplectic groupoids and {P}oisson
  manifolds.}
{\em Bull. Amer. Math. Soc. (N.S.)}  {\bf 16} (1987), 101--104.


\bibitem{We88}
{Weinstein, A.}: Coisotropic calculus and Poisson groupoids.
{\em J. Math. Soc. Japan} {\bf 40} (1988), 705-727.

\bibitem{YI}
Yano, K., Ishihara, S.:{\em Tangent and cotangent bundles}, Marcel
Dekker, Inc., New York, 1973.

\end{thebibliography}
\end{document}